\documentclass[11pt,a4paper,reqno, draft]{amsart}
\usepackage{amsmath, amssymb, enumerate,  graphicx,tikz, stmaryrd, eufrak,enumitem,MnSymbol}

\usepackage{thmtools}

\usepackage[pdftex]{hyperref}
\usepackage{cleveref}
\usepackage{bbm}
\usepackage{cases}
\usepackage{array}
\usepackage{tikz-cd}
\usepackage[all]{xy}

\usepackage[hmarginratio={1:1},vmarginratio={1:1},lmargin=80.0pt,tmargin=90.0pt]{geometry}

\usepackage{tikz}
\tikzset{anchorbase/.style={baseline={([yshift=-0.5ex]current bounding box.center)}}}
\usetikzlibrary{decorations.markings}
\usetikzlibrary{decorations.pathreplacing}
\usetikzlibrary{arrows,shapes,positioning,backgrounds}
\tikzstyle directed=[postaction={decorate,decoration={markings,
    mark=at position #1 with {\arrow{>}}}}]
\tikzstyle rdirected=[postaction={decorate,decoration={markings,
    mark=at position #1 with {\arrow{<}}}}]
    
\usepackage{graphicx}
\usepackage[vcentermath]{youngtab}
\usepackage{nicefrac}

\numberwithin{equation}{section}

\newtheorem{theorem}[subsubsection]{Theorem}
\newtheorem{lemma}[theorem]{Lemma}
\newtheorem{prop}[theorem]{Proposition}
\newtheorem{corollary}[subsubsection]{Corollary}
\newtheorem{porism}[subsubsection]{Porism}
\newtheorem{conjecture}[theorem]{Conjecture}

\theoremstyle{definition}
\newtheorem{definition}[subsubsection]{Definition}

\newtheorem{remark}[theorem]{Remark}

\newtheorem{example}[subsubsection]{Example}

\newtheorem{question}[theorem]{Question}

\newcommand{\m}{\mathfrak{m}}
\newcommand{\fm}{\mathfrak{m}}
\newcommand{\p}{\mathfrak{p}}
\newcommand{\fp}{\mathfrak{p}}
\newcommand{\q}{\mathfrak{q}}
\newcommand{\fq}{\mathfrak{q}}

\newcommand{\bA}{\mathcal{A}}

\newcommand{\bC}{\mathcal{C}}

\newcommand{\cL}{\mathcal{L}}
\newcommand{\cD}{\mathcal{D}}

\newcommand{\co}{\mathrm{co}}

\newcommand{\red}{\mathrm{red}}

\newcommand{\Aff}{\mathsf{Aff}}

\newcommand{\Fais}{\mathsf{Fais}}

\newcommand{\Fun}{\mathsf{Fun}}
\newcommand{\Alg}{\mathsf{Alg}}
\newcommand{\LocAlg}{\mathsf{LocAlg}}
\newcommand{\Set}{\mathsf{Set}}
\newcommand{\Grp}{\mathsf{Grp}}

\newcommand{\Mod}{\mathsf{Mod}}
\newcommand{\Rep}{\mathsf{Rep}}
\newcommand{\Tilt}{\mathsf{Tilt}}
\newcommand{\QCoh}{\mathsf{QCoh}}

\newcommand{\gr}{\mathrm{gr}}
\newcommand{\Frac}{\mathrm{Frac}}

\newcommand{\tto}{\twoheadrightarrow}
\newcommand{\cO}{\mathcal{O}}

\newcommand{\mN}{\mathbb{N}}
\newcommand{\mG}{\mathbb{G}}
\newcommand{\mZ}{\mathbb{Z}}
\newcommand{\mA}{\mathbb{A}}

\newcommand{\mC}{\mathbb{C}}

\newcommand{\End}{\mathrm{End}}
\newcommand{\im}{\mathrm{im}}

\newcommand{\Ext}{\mathrm{Ext}}
\newcommand{\Tor}{\mathrm{Tor}}

\newcommand{\Hom}{\mathrm{Hom}}
\newcommand{\Ann}{\mathrm{Ann}}
\newcommand{\Aut}{\mathrm{Aut}}

\newcommand{\Sym}{\mathrm{Sym}}

\newcommand{\Lie}{\mathrm{Lie}}

\newcommand{\op}{\mathrm{op}}
\newcommand{\Ind}{\mathrm{Ind}}
\newcommand{\Res}{\mathrm{Res}}

\newcommand{\Spec}{\mathrm{Spec}}
\newcommand{\GSpec}{\mathcal{S}pec}
\newcommand{\Nil}{\mathrm{Nil}}
\newcommand{\Id}{\mathrm{Id}}
\newcommand{\Vecc}{\mathsf{Vec}}
\newcommand{\sVec}{\mathsf{sVec}}
\newcommand{\Ver}{\mathsf{Ver}}
\newcommand{\Top}{\mathsf{Top}}

\newcommand{\unit}{{\mathbbm{1}}}
\newcommand{\cC}{\mathcal{C}}

\begin{document}
\title[Algebra in tensor categories]{Commutative algebra in tensor categories}
\author{Kevin Coulembier}
\address{K.C.: School of Mathematics and Statistics, University of Sydney, NSW 2006, Australia}
\email{kevin.coulembier@sydney.edu.au}


\subjclass[2020]{}

\begin{abstract}
We develop some foundations of commutative algebra, with a view towards algebraic geometry, in symmetric tensor categories. Most results establish analogues of classical theorems, in tensor categories which admit a tensor functor to some tensor category verifying specific conditions. This is in line with the current program which aims to describe tensor categories by their tensor functors to incompressible categories. We place particular emphasis on the notion of observable subgroups of affine group schemes in tensor categories, which in particular leads to some further insight into observability for classical affine group schemes.
\end{abstract}

\maketitle


\section*{Introduction}

Tensor categories (always assumed to be symmetric in this paper) over fields of characteristic zero are well-understood thanks to the classical work of Deligne \cite{Del90, Del02}, although the exotic case of tensor categories which are not of moderate growth continues to be an exciting research area, see \cite{HS}.
Recent breakthroughs in the theory of tensor categories of moderate growth over fields of positive characteristic, see for instance \cite{BE, BEO, Tann, CEO, CEO2, Os}, have revealed the richness of the theory and the importance of `incompressible tensor categories'. Note that in characteristic zero, the only incompressible categories of moderate growth are the categories of (super) vector spaces.

These developments motivate the following three directions for research. One would expect each incompressible category to yield a well-behaved theory of algebraic geometry, just as the category of vector spaces connects to ordinary algebraic geometry (over a field) and the category of super vector spaces gives rise to super geometry. Secondly, the incompressible categories $\Ver_{p^n}$ from \cite{BE, BEO, AbEnv} which have been discovered thus far, are not yet well-understood in certain aspects. Thirdly, the major open question in the area is whether the current list of incompressible tensor categories is exhaustive.

In this paper we develop some foundational theory of commutative algebra in tensor categories, as the canonical precursor to the first direction identified above (algebraic geometry). However, this neatly ties in with the other two lines of research. For instance, we identify two sufficient conditions on tensor categories, and call the categories that satisfy them `maximally nilpotent' ({\bf MN}) and `geometrically reductive' ({\bf GR}), for validity of paramount classical properties in commutative algebra in those categories. As one would hope (based on the desire to develop algebraic geometry specifically in incompressible categories), these potential properties seem closely related to whether a category is incompressible. This is explored in the follow-up paper \cite{CEO2}. This in turn motivates the question of whether they are satisfied in~$\Ver_{p^n}$. Partial positive answers to these questions are obtained in the current paper as well as in \cite{CEO2}.

Now we describe the precise content of the paper. In Section~\ref{SecPrel} we recall some necessary background. In Section~\ref{SecDefs} we formulate the basic definitions for commutative algebra. In Section~\ref{SecNPGR} we introduce and study {\bf MN} and {\bf GR} tensor categories. In Section~\ref{SecLoc} we study the principle of localisation of algebras in tensor categories. In Section~\ref{SecComp} we show how the previous results relate to a comparison between the spectra of algebras in tensor categories and their invariant subalgebras. In Section~\ref{SecHer} we investigate some potential properties which are inherited by any tensor category which admits a tensor functor to a category in which they are valid. These include classical properties relating to Nakayma's lemma, the fact that commutative Hopf algebras are faithfully flat over Hopf subalgebras and the fact that finitely generated algebras are noetherian. The conclusion is invariably that these properties are satisfied in categories which are both {\bf MN} and {\bf GR}, demonstrating further the importance of proving that incompressible categories are {\bf MN} and {\bf GR}.


Finally, in Section~\ref{SecGrp} we study affine group schemes in tensor categories. Mostly we verify standard properties of quotients and normal subgroups, but we pay special attention to `observability'. This is motivated by the important role the classical theory of observable subgroups, see~\cite{BBHM, Gr}, has played recently  in studying tannakian categories appearing as abelian envelopes in \cite{CEOP}. Our approach actually differs from the more classical ones and seems to give new information (and new proofs) even for ordinary affine group schemes over fields, for instance removing conditions like smoothness, finite type and algebraically closed fields. This application is written out in Appendix~\ref{App}. This appendix also highlights that the existence of a well-behaved theory of algebraic geometry (in this case based on $\Vecc$) implies that the `main theorem' of observability (a priori not related to geometry) is satisfied, yielding further  concrete motivation to develop further algebraic geometry in other tensor categories.

\section{Preliminaries}\label{SecPrel}

Let $k$ be a field.

\subsection{Tensor categories}

We refer to \cite{Del90, DM, EGNO} for a full introduction to the theory of tensor categories.

\subsubsection{}Following \cite{Del90}, an essentially small $k$-linear symmetric monoidal category $(\bC,\otimes,\unit)$ is a {\bf tensor category over $k$} if
\begin{enumerate}
\item $\bC$ is abelian;
\item $k\to\End(\unit)$ is an isomorphism;
\item $(\bC,\otimes,\unit)$ is rigid, meaning that every object $X$ has a monoidal dual $X^\vee$.
\end{enumerate}
By \cite[Proposition~1.17]{DM}, it then follows that $\unit$ is simple.
If additionally, we have
\begin{enumerate}
\item[(4)] Every object in~$\bC$ has finite length;
\end{enumerate}
then using (2) and (3) shows also that morphism spaces are finite dimensional. Tensor categories satisfying (4) will be called {\bf pretannakian categories}.

The important advantage of pretannakian categories that we will rely on is that ind-objects are equal to the union (direct limit) of compact subobjects. In fact, pretannakian categories are (when ignoring the tensor product) equivalent to categories of finite dimensional comodules of coalgebras.

Under condition (4) the following condition is well-defined:
\begin{enumerate}
\item[(5)] For every object $X\in\bC$, the function
$$\mN\to\mN,\quad n\mapsto \ell(X^{\otimes n})$$
is bounded by an exponential.
\end{enumerate}
Pretannakian categories satisfying (5) are called (tensor categories) of {\bf moderate growth}.

A {\bf tensor functor} between tensor categories is a $k$-linear exact symmetric monoidal functor. It is well-known, see \cite{DM}, that tensor functors are automatically faithful. In fact, one can interchange the words `exact' and `faithful' in the above definition of tensor functor, see~\cite{CEOP}.

\subsubsection{} We will make little distinction between a tensor category $\bC$ and its ind-completion $\Ind\bC$. Similarly, we will use the same notation for a tensor functor and its cocontinuous extension to the ind-completions.


It will be important for the sequel that, for a tensor category $\bC$, the ind-completion $\Ind\cC$ is complete, as a Grothendieck category. However, limits need not be well behaved with respect to tensor functors, see Example~\ref{ExInverse} below.

\subsubsection{} A {\bf tensor subcategory} of a tensor category is a topologising subcategory closed under the tensor product and duals. A tensor category $\cC_1$ is {\bf finitely generated} over a tensor subcategory $\cC\subset\cC_1$ if there is some $X\in \cC_1$ so that the only tensor subcategory of $\cC_1$ which contains both $\cC$ and $X$ is $\cC_1$ itself. In particular, a tensor category is called finitely generated if it is finitely generated over $\Vecc$.

%
%
%
%
%
%

\subsection{Affine group schemes}

\subsubsection{}\label{ExTann}
For an affine group scheme~$G$ over $k$, we have the (pre)tannakian tensor category $\Rep G$ of finite dimensional rational $G$-representations. In this case the ind-completion is the category $\Rep^\infty G$ of all rational $G$-representations. We will always use the symbol $\omega$ to denote the forgetful tensor functor
$$\omega: \Rep G\to\Vecc\quad\mbox{ or }\quad\omega: \Rep^\infty G\to\Vecc^\infty$$
to the category of (finite dimensional) vector spaces.

We will use some running examples, for which we fix some terminology:

\subsubsection{Infinitesimal group schemes}

\label{DefInf} For a (non-trivial) finite affine group scheme~$G$ over a field $k$ (meaning the coordinate algebra $\cO(G)$ is finite dimensional), the following conditions are equivalent:
\begin{enumerate}
\item The coordinate algebra $\cO(G)$ is local.
\item $G$ is connected.
\item The coordinate algebra $\cO(G)$ has as spectrum one point.
\item The field $k$ has characteristic $p>0$, and there is a power $q$ of $p$ for which 
$$(f-\varepsilon(f))^q\;=\;0\;=\; f^q-\varepsilon(f^q),\quad\mbox{ for all $f\in \cO(G)$.}$$
\end{enumerate}
Such a finite group scheme is called {\bf infinitesimal}.

\subsubsection{Additive group}\label{Ga}
We will mainly consider the additive group $\mG_a$ over $\mC$. The indecomposable objects in~$\Rep\mG_a$ are labelled by their dimension as $\{M_i\,|\, i\in\mZ_{>0}\}$.

We fix a non-zero $D\in \Lie\mG_a$, so that a $\mG_a$-representation `is' as a vector space with a nilpotent endomorphism (the action of $D$).

\begin{example}\label{ExInverse}
Set $k=\mC$ and $\cC=\Rep\mG_a$.  We choose a basis $\{e_{ij}\,|\, 1\le j\le i\}$ of $\omega(M_i)$ compatible with the socle filtration (e.g. $\mC e_{i1}$ is the socle).  In~$\Rep^\infty\mG_a$, the limit
$$\prod_{i>0} M_i\;\simeq\; \varprojlim_j \bigoplus_{i=1}^j M_i$$
is given by the subspace of $\prod_i \omega(M_i)$ consisting of the vectors
$\left(\sum_j a_{ij} e_{ij}\right)_{i}$
for which there exists $l\in\mN$ such that $a_{ij}=0$ when $j>l$, with natural $\mG_a$-action. This can be computed directly, or by using the fully faithful embedding $F:\Rep\mG_a\to\Rep\Lie\mG_a$. Indeed $F_\ast$ (which is defined in \ref{Fs} below and acts by taking the maximal integrable submodule) acting on the product of the $M_i$ inside $\Ind\Rep\Lie\mG_a$ (which is just the product of vector spaces) yields the above space.
\end{example}

\subsection{The right adjoint of a tensor functor}

\subsubsection{}\label{Fs} A tensor functor $F:\cC\to\cD$ has, by the special adjoint functor theorem, always a right adjoint
$$F\dashv F_\ast:\Ind\cD\to\Ind\cC.$$
This functor  $F_\ast$ has a canonical lax-monoidal structure, is continuous and commutes with filtered colimits.

\subsubsection{}
For the inclusion $I:\Vecc\hookrightarrow \cC$, we denote $I_\ast$ by
$$H^0\simeq\Hom(\unit, -):\Ind\cC\to\Vecc^\infty,$$
the functor returning the maximal subobject belonging to $\Vecc^\infty\subset \Ind\cC$.
Note that unless $\bC$ is pretannakian, $H^0$ does not need to send $\bC$ to $\Vecc$. We will typically abbreviate $H^0(X)$ to $X^0$.
In Example~\ref{ExTann}, we have 
$$H^0(X)=:X^0=X^G:=H^0(G,X),$$ the space of $G$-invariants in a representation $X$ of $G$.

\subsubsection{}\label{deftilde}
The inclusion $I:\Vecc\hookrightarrow \cC$ also has a left adjoint 
$$H_0:\Ind\cC\to\Vecc^\infty,\quad X\mapsto X/\widetilde{X}\;\mbox{ with }\; \widetilde{X}=\bigcap_{f\in\Hom(X,\unit)}\ker f.$$
If $\cC$ is pretannakian and $X\in \cC$, then $H_0(X)^\ast\simeq H^0(X^\vee)$.

\begin{prop}\label{Propadj0}
The following conditions are equivalent on a tensor functor $F:\cC\to\cD$ between pretannakian categories:
\begin{enumerate}
\item $F_\ast$ is faithful;
\item Every object in~$\cD$ is a quotient of an object $F(X)$, for $X\in\cC$;
\item Every object in~$\cD$ is a subobject of an object $F(X)$, for $X\in\cC$.
\end{enumerate}
\end{prop}
\begin{proof}
Conditions (2) and (3) are equivalent by duality.

We prove that (1) implies (2). By adjunction, (1) implies that $\epsilon_Y:FF_\ast Y\to Y$ is an epimorphism for every $Y\in\cD$. As $Y$ is compact, it follows easily there is a compact subobject $X\subset F_\ast Y$ such that $FX\to Y$ is still an epimorphism.

That (2) implies (1) can be proved easily directly, or again via $\epsilon_Y$.
\end{proof}



\subsection{Higher Verlinde categories}

The motivation behind most of the current paper is to develop algebraic geometry in tensor categories like to the categories $\Ver_{p^n}$ (defined over fields of characteristic $p>0$) from \cite{BEO}, to which we refer for definition and properties. In the current paper we work in more abstraction and generality. We do investigate to some extent validity of hypotheses in the categories $\Ver_{p^n}$. We also use $\Ver_{p^\infty}$ for the union of the categories $\Ver_{p^n},n>0$ as constructed in \cite{BE, BEO}. For $\mathrm{char}(k)\not=2$, we define the tensor category $\sVec$ of supervector spaces as the category of finite dimensional $\mZ/2$-graded vector spaces with Koszul sign rule for the braiding.
We use the notational shortcut 
$$\Ver_{k}\;=\;\begin{cases}
\sVec_{k}&\mbox{ if $\mathrm{char}(k)=0$}\\
\Ver_p&\mbox{ if $\mathrm{char}(k)=p>0$}
\end{cases}$$
For example, in characteristic $2$, we have $\Ver_{k}=\Vecc_k$.
For some justification of $\Ver_p$ as a positive characteristic analogue of $\sVec$, see \cite{CEO, Os}.

\section{The basics of commutative algebra}\label{SecDefs}
For the rest of the paper we let $\cC$ be a pretannakian category.

\subsection{Algebras}

\subsubsection{}
Denote by $\Alg\cC$ the category of commutative algebras (symmetric ring objects) in~$\Ind\cC$. Such a ring object $(A,\mu,\eta)$, with $\mu:A\otimes A\to A$ and $\eta:\unit\to A$ will usually simply be denoted by the underlying object $A$.

For subobjects $X\subset A\supset Y$, we write $XY=YX$ for the image of $X\otimes Y\to A$. We also write $X^n$ for $XX\cdots X$.
For $A\in\Alg\cC$, an {\bf ideal} $I<A$ is a subobject $I\subset A$ such that $I A\subset I$. An algebra $A$ is {\bf simple} if it has precisely two ideals, $0$ and $A$.

For convenience, we point out that an algebra morphism $f: A\to B$ \ is a monomorphism in~$\Alg\cC$ if and only if it is a monomorphism in~$\Ind\cC$.

\subsubsection{}\label{A0A0}
 By the lax monoidal structure of $H^0$, for an algebra $A$ in~$\Ind\cC$, $A^0$ is a subalgebra. In fact, we will often use the obvious algebra isomorphism
 \begin{equation}\label{EndA}
 A^0\;\simeq\; \End_A(A).
 \end{equation}
 The assignment $A\mapsto A^0$ is the right adjoint of the canonical inclusion $\Alg_k\to\Alg\cC$. 
 
 With $\widetilde{A}$ as in \ref{deftilde}, the left adjoint of this inclusion is given by
 $$A\mapsto A_0:=A/R(A),\quad\mbox{with } \; R(A):=A\widetilde{A}.$$

\subsubsection{}The category of $A$-modules in~$\Ind\bC$ is denoted by $\Mod_{\bC}A$. In the particular case where $A$ is an ordinary $k$-algebra interpreted to be in~$\Ind\bC$, an $A$-module in~$\bC$ is the same as an object $M\in\Ind\bC$ equipped with algebra morphism $A\to\End(M)$.

%

For an $A$-module $M$, its annihilator ideal $\Ann_A(M)$ is the maximal ideal $I<A$ for which $I\otimes M\to M$ is zero. Equivalently, it is the kernel of the morphism from $A$ to the internal hom from $M$ to itself.

For a morphism $f:X\to Y$ we have the associated morphism of graded algebras $\Sym X\to \Sym Y$ and we denote the corresponding morphisms by $f^n:\Sym^nX\to\Sym^nY$.

\subsection{Ideals}
Consider $A\in\Alg\cC$.
\subsubsection{Prime ideals}

An ideal $I<A$ is {\bf prime} if for every $\bC\ni X,Y\subset A$, the condition $XY\subset I$ implies that at least one of $X\subset I$ or $Y\subset I$ is true. We say that $A$ is a {\bf domain} if $0$ is a prime ideal. We reserve the term `integral domain' for domains in~$\Vecc^\infty$.

We will typically denote prime ideals by $\p$ (or $\q$) and maximal ideals (which are automatically prime) by $\m$.

We denote by $\Spec A$ the set of prime ideals in~$A$. For any ideal $J<A$ we have 
$$V(J)\;=\;\{\p\in\Spec A\,|\, J\subset \p\}\;\subset\; \Spec A$$
and these are the closed subsets of the {\bf Zariski topology} on $\Spec A$. We then obtain a functor
$$\Spec:\; (\Alg\bC)^{\op}\;\to\; \Top.$$


\begin{remark}
One can easily verify that $\Spec A$ is a `spectral space' and $\Spec f$ is a `spectral map' for a $f:A\to B$ in~$\Alg\cC$, in the terminology of \cite{Ho}.
\end{remark}

\subsubsection{Radical ideals} For an ideal $I<A$, we denote by $\sqrt{I}$ the intersection of all prime ideals which contain~$I$, so
$$\sqrt{I}=\cap_{\p\in V(I)}\p.$$ An ideal $I$ is {\bf radical} if $I=\sqrt{I}$. As observed in Corollary~\ref{CorNil}(2) below, the more classical definition of a radical ideal is equivalent.

The assignment $J\mapsto V(J)$ yields a bijection (an isomorphism of posets for obvious partial orders) between the set of radical ideals and the set of closed subsets of $\Spec A$.


\subsubsection{Nil ideals}
A subobject $ X\subset A$ is {\bf nilpotent} if $X^n=0$ for some $n\in\mN$. An ideal $I<A$ is {\bf nil} if every $\bC\ni X\subset I$ is nilpotent. The unique maximal nil ideal (generated by all nilpotent subobjects) is denoted by $\Nil A$ and called the nilradical. A finitely generated nil ideal is nilpotent.

\begin{lemma}\label{LemNil}
The nilradical $\Nil A$ is the intersection of all prime ideals, so $\Nil A=\sqrt{0}$. 
\end{lemma}
\begin{proof}
The inclusion $\Nil A\subset \fp$ for every prime ideal $\fp$ is obvious. If the resulting inclusion $\Nil A\subset \cap_\fp\fp$ (union over all prime ideals) were not an equality, it would be possible to find $\cC\ni X\subset \cap_\fp\fp$ with $X^n\not=0$ for all $n\in\mN$.

 To complete the proof, we thus consider $\cC\ni X\subset A$ and assume $X^n\not=0$ for all $n\in\mN$. We need to show that there exists a prime ideal which does not contain~$X$.

 Consider the set $S$ of ideals $I<A$ which contain none of the $X^n\subset A$. By assumption $0\in S$. Let $P$ be a maximal element in~$S$, which exists by application of Zorn's lemma. Now $P$ must be prime. Indeed, for two ideals $I,J$ which are not contained in~$P$, we must have $X^a\subset P+I$ and $X^b\subset P+J$ for some $a,b$ and hence $X^{a+b}\subset P+IJ$. Therefore $IJ\not\subset P$.
Since, by definition, $X$ is not contained in~$P$, this concludes the proof.
\end{proof}

We have some immediate consequences of the lemma.
\begin{corollary}
The following are equivalent for $X\in\bC$:
\begin{enumerate}
\item $\Sym X$ is finite ({\it i.e.} $\Sym^nX=0$ for some $n\in\mN$);
\item $\Sym X$ is infinitesimal ({\it i.e.} $\Spec\Sym X$ is a singleton).
\end{enumerate}
\end{corollary}

\begin{corollary}\label{CorNil}
\begin{enumerate}
\item For an ideal $I<A$, the radical $\sqrt{I}$ is equal to the ideal spanned by all $\bC\ni X\subset A$ for which $X^n\subset I$ for some $n\in\mN$.
\item An ideal $I<A$ is radical if and only if for every $ \cC\ni X\subset A$ and $n\in \mN$, the condition $X^n\subset I$ implies $X\subset I$.
\item We have $(\sqrt{I})^0=\sqrt{I^0}$.
\end{enumerate}
\end{corollary}
\begin{proof}
Parts (2) and (3) are immediate consequences of part (1). For part (1), it is obvious that $X^n\subset I$ implies $X\subset\sqrt{I}$. Conversely, consider some $\cC\ni X\subset\sqrt{I}$. By Lemma~\ref{LemNil}, the image of $X$ in~$A/I$ is nilpotent, from which it follows that $X^n\subset I$ for some $n$.
\end{proof}

\begin{corollary}\label{CorDense}
For a monomorphism $i:A\hookrightarrow B$ in~$\Alg\cC$, the image of $\Spec B\to \Spec A$ is dense.
\end{corollary}
\begin{proof}
Clearly, the closure of any subset of prime ideals is given by $V(J)$ for $J$ the intersection of the set of prime ideals. By Lemma~\ref{LemNil}, the closure of the image of $\Spec(i)$ is thus $V(J)$ for $J=A\cap \Nil B$. We have $J= \Nil A$ and $V(\Nil A)=\Spec A$ by definition of $\Nil$.
\end{proof}

\subsection{Algebra morphisms}
Consider $A\in\Alg\cC$.
\begin{definition}\label{defff}
An $A$-module $M$ is {\bf flat}, if the endofunctor
$M\otimes_A-$ of $\Mod_{\cC}A$ is exact and {\bf faithful} if the functor is faithful.
\end{definition}

\begin{remark}\label{RemFaith}
Consider a tensor functor $F:\cC\to\cC_1$ and $A\in\Alg\cC$ with a module $N$. If $FN$ is faithful (resp. flat) as an $FA$-module, then also the $A$-module $N$ is faithful (resp. flat).
\end{remark}

We will only need the following proposition in the classical case ($\cC=\Vecc$), but we can easily extend it to semisimple pretannakian categories.

\begin{prop}(Govorov - Lazard)
Assume that $\cC$ is semisimple. Then $M\in\Mod_{\cC}A$ is flat if and only if it is a direct limit of $A$-modules $A\otimes X$, $X\in\cC$.
\end{prop}
\begin{proof}
One direction is obvious. For the other direction, the proof reduces quickly to the proving the claim that any diagram of $A$-module morphisms
$A\otimes X\to A\otimes Y\to M,$
which composes to zero, can be completed to a commutative diagram
$$\xymatrix{
A\otimes X\ar[r]& A\otimes Y\ar[r]\ar[rd]& A\otimes Z\ar[d]\\
&&M
}$$
for some $Z\in \cC$ such that the top row composes to zero. We start by completing the adjoint of $A\otimes X\to A\otimes Y$ to an exact (in the middle) sequence of $A$-modules
$$\bigoplus_iA\otimes Z_i^\vee\;\to\; A\otimes Y^\vee\;\to\; A\otimes X^\vee.$$
Since $M$ is flat, applying $-\otimes_AM$ yields an exact sequence
$$\bigoplus_i Z_i^\vee\otimes M\;\to\;Y^\vee\otimes M\;\to\; X^\vee\otimes M.$$
Now $A\otimes X\to M$ corresponds to a morphism $\unit\to Y^\vee\otimes M$ which is sent to zero under $Y^\vee\otimes M\to X^\vee\otimes M$. We can thus let it factor through (some finite sub-sum of) $(\oplus_i Z_i^\vee)\otimes M$. Applying adjunction again yields the desired $A\otimes (\oplus_iZ_i)\to M$.
\end{proof}

As a special case of Definition~\ref{defff}, an algebra morphism $f:A\to B$  is {flat} if the monoidal functor
$$B\otimes_A-\,:\, \Mod_{\cC}A\to\Mod_{\cC}B$$
is exact, and {faithful} if the functor is faithful.

\begin{remark}\label{epifp}
As in the classical case, for a faithfully flat algebra morphism $A\to B$, it follows that we have an equaliser (in~$\Ind\cC$)
$$A\to B\rightrightarrows B\otimes_AB,$$
by splitting the sequence via $B\otimes_A-$. In particular, if $A\to B$ is an epimorphism as well as faithfully flat, it must be an isomorphism.
\end{remark}

\begin{lemma}\label{FaithMax}
Take $A\in\Alg\cC$ and an exact monoidal functor $F:\Mod_{\cC}A\to\bA$ to an abelian monoidal category~$\bA$. Then $F$ is faithful if and only if $F(A/\m)\not=0$ for all maximal ideals~$\m<A$.
\end{lemma}
\begin{proof}
One direction is obvious. To prove the other direction by contradiction, take an $A$-module $N$ and assume that $F(N)=0$. Take some subobject $\cC\ni X\subset N$. By assumption,
$$F(X^\vee\otimes N)\simeq F((A\otimes X^\vee)\otimes_AN)\simeq F(A\otimes X^\vee)\otimes F(N)=0.$$

However, $N\otimes X^\vee$ is an $A$-module with non-zero $A$-module morphism $A\to N\otimes X^\vee$. Let $I$ be the kernel, then $F(A/I)=0$, so also $F(A/\m)$ for any maximal ideal which contains $I$, a contradiction. 
\end{proof}

\subsubsection{}\label{Deffp}
An algebra morphism $A\to B$ in~$\Alg\cC$ is {\bf finitely presented} in the categorical sense (i.e. $B$ is compact in the category of $A$-algebras), if and only if there exists $X\in\cC$ such that~$B$ is the quotient of $A\otimes\Sym X$ by an ideal generated by a subobject in~$\cC$.

\subsection{Grothendieck topologies}

\subsubsection{}

As in the classical case, we can equip the category $\Aff\cC:=(\Alg\cC)^{\op}$ with various Grothendieck topologies (in order to work with an essentially small category, we can restrict the sizes of the allowed algebras).

\begin{definition}
The covers in the {\bf fpqc pretopology} on $\Aff\cC$ correspond to the finite collections of algebra morphisms $\{A\to B_i\,|\, i\in I\}$ for which $A\to \prod_i B_i$ is faithfully flat.
\end{definition}

We denote by 
$$\Fais\cC \subset \Fun\cC$$
the category of sheaves for the fpqc topology. Concretely, $F\in \Fun\cC$ belongs to $\Fais\cC$ if and only if $F$ commutes with finite products and if 
$$F(A)\to F(B)\rightrightarrows F(B\otimes_AB)$$
is an equaliser for every faithfully flat algebra morphism $A\to B$.

\begin{lemma}\label{LemDodu}
Consider $F\subset G\in\Fun\cC$ such that $G\in \Fais\cC$ and $F$ commutes with finite products and is such that $F(A)\to F(B)$ is a monomorphism whenever $A\to B$ is faithfully flat in~$\Alg\cC$. Then the following are equivalent:
\begin{enumerate}
\item $F\to G$ is the sheafification of $F$;
\item For every $x\in G(A)$ there exists a faithfully flat $A\to B$ and $y\in F(B)$ such that the images in~$G(B)$ of $x$ and $y$ coincide.
\end{enumerate}  
\end{lemma}
\begin{proof}
These are standard properties. That (2) implies (1) is a special case of \cite[III. \S 1.1.7]{DG}. That (1) implies (2) is a consequence of the realisation of sheafification of $F$ as a filtered colimit over all faithfully flat $A\to B$ in~\cite[III. \S 1]{DG}.
\end{proof}


\section{Some convenient assumptions}\label{SecNPGR}
Recall that $\cC$ denotes some pretannakian category.
In this section we introduce some potential properties of $\cC$. Throughout the paper, as well as in \cite{CEO2}, we will demonstrate that they have far-reaching implications. In this section we also give examples of cases where the assumptions are satisfied. In the next section, we demonstrate how the conditions can be formulated entirely in terms of spectra of algebras.

\subsection{Geometric reductivity}
Consider the following potential property of $\cC$:
\begin{enumerate}
\item[{\bf(GR)}] For every non-zero morphism $X\tto\unit$ in~$\cC$, there exists $n>0$ for which $\Sym^n X\tto\unit$ is split.
\end{enumerate}
If this property is satisfied we say that $\cC$ is {\bf geometrically reductive} or simply {\bf GR}.

\begin{example}
If $\cC=\Rep G$ for an affine group scheme $G$, then we can reformulate condition~{\bf (GR)} as follows. For every representation $V$ with 1-dimensional trivial subrepresentation $L\subset V$, there exist $n>0$ and $G$-invariant $f\in S^n(V^\ast)$ (a homogeneous polynomial on $\mA V$) with $f(L)\not=0$, in other words $G$ is {geometrically reductive}.
\end{example}

It thus follows from Haboush's theorem (Mumford's conjecture) that representation categories of reductive groups are {\bf GR}. In particular, the generalisation in \cite{Wa} implies:

\begin{corollary}[Haboush, Waterhouse]\label{CorGRG}
If $k$ is a perfect field and $G$ an algebraic affine group scheme over $k$, then $\Rep G$ is {\bf GR} if and only if $G^0_{\rm red}$ is reductive.
\end{corollary}

We prove a special case of the corollary directly, for arbitrary fields.
\begin{lemma}
If $G$ is an infinitesimal group scheme, then $\Rep G$ is {\bf GR}.
\end{lemma}
\begin{proof}
For a representation $V$ with non-zero $V\tto k$, we choose a basis $\{v_1,\cdots, v_n\}$ of the kernel and some $v\in V$ not in the kernel. We write the coaction of $\cO(G)$ on $v$ as
$$\rho(v)\;=\; 1\otimes v+\sum_i f_i\otimes v_i.$$
The fact that $\varepsilon(f_i)=0$ shows $f_i^q=\varepsilon(f_i^{q})=0$, for a power $q$ of $p$ as in \ref{DefInf}(3). It follows that $v^q\in \Sym^q V$ spans a subrepresentation which splits $\Sym^qV\tto k$.
\end{proof}

\begin{lemma}\label{LemGR}
 $\cC$ is {\bf GR} if and only if:
 
 {\rm If $\mathrm{char}(k)=0$: } $\cC$ is semisimple.
 
{\rm If $\mathrm{char}(k)=p>0$: } For every non-zero morphism $X\tto\unit$ in~$\cC$, there exists $j>0$ for which $\Sym^{p^j} X\tto\unit$ is split.

\end{lemma}
\begin{proof}
For part (1), $\Sym^n X$ is a direct summand of $\otimes^n X$, condition {\bf (GR)} guarantees a section $\gamma$ to $\otimes^nX\xrightarrow{\otimes^n \beta}\unit$ for every $\beta:X\tto\unit$. Considering the composition
$$\unit\xrightarrow{\gamma}\otimes^n X\xrightarrow{\otimes^{n-1} \beta \otimes X}X$$
yields a section for $\beta$. Hence, $\Ext^1(\unit,-)$ vanishes and by adjunction $\Ext^1$ is identically zero.

Part (2) is proved similarly, by writing $\Sym^n X$ as a direct summand of
$$\bigotimes_j(\Sym^{p^j}X)^{\otimes n_j},$$
for the $p$-adic expansion $n=\sum_jn_jp^j$.
\end{proof}


\begin{theorem}\label{ThmGR}
\begin{enumerate}
\item If for every finitely generated $A\in\Alg\cC$, the subalgebra $A^0\in\Alg_k$ is finitely generated, then $\cC$ is {\bf GR}.
\item $\cC$ is {\bf GR} if and only if for every $A\in \Alg \cC$ the following is satisfied:

{\rm If $\mathrm{char}(k)=0$:} The morphism $A^0\to A_0$ is a surjection;

{\rm If $\mathrm{char}(k)=p>0$:} For every $a\in A_0$, there exists $j\in\mZ_{>0}$ for which $a^{p^j}$ is in the image of $A^0\hookrightarrow A\tto A_0$.

\end{enumerate}
\end{theorem}
\begin{proof}
 All ideas for proving part (1) can be found in \cite[\S 2]{vdK}.
Concretely, for $\alpha:X\tto\unit$ in~$\cC$ consider the algebra $B=\Sym X$ and the ideal $J<B$ generated by $\ker \alpha$ in degree 1. Let~$A$ denote the algebra
$$A:= B\oplus B/J\simeq (\Sym X)\oplus  \unit^{\aleph_0},$$
where $B$ is a (unital) subalgebra and $B/J$ is an ideal which squares to zero and the multiplication $B\otimes B/J\to B/J$ comes from the ordinary multiplication. In particular
$$A^0\;\simeq\; B^0\oplus k^{\aleph_0}.$$

Assume for a contradiction that no $\alpha^n$ for $n>0$ admits a section. This is equivalent to saying that the composite $B^0\to B\to B/J\simeq k[x]$ is zero, except in degree $0$. Consequently, the multiplication of the augmentation ideal in~$B^0$ on $k^{\aleph_0}$ is zero, so every subspace of $k^{\aleph_0}$ is an ideal of $A^0$. Clearly,~$A^0$ is not noetherian as a $k$-algebra, yielding a contradiction.

For part (2), we restrict to the case $p>0$, with the characteristic zero case being easier. Assume first that {\bf (GR)} is satisfied. For $a\in A_0$, we can consider the pullback of $A\tto A_0\leftarrow \unit$, by compactness of $\unit$, there exists $\cC\ni X\subset A$ with simple top $\unit$ so that the image of $X$ under $A\tto A_0$ is spanned by $a$. We can then use the assumed splitting of $\Sym^{p^j} X\tto\unit$ in Lemma~\ref{LemGR}(2) to construct $\unit\to\Sym^{p^j} X\to A$ which maps to $a^{p^j}$. 

Now assume that the property in (2) is satisfied for all $A$. For $X\tto \unit$ we can consider
$$A\tto A_0\tto k[x],$$
where we realise $k[x]$ as the quotient of $A:=\Sym X$ by the ideal generated in degree 1 by the kernel of $X\tto \unit$. The assumption applied to an element in the preimage of $x$ in~$A_0$ then allows us to construct the desired $\unit\to\Sym^nX$.
\end{proof}

\begin{conjecture}\label{ConjNag}
\begin{enumerate}
\item Nagata's theorem extends (from tannakian categories, see \cite{BF}) to arbitrary tensor categories of moderate growth, {\it i.e.} Theorem~\ref{ThmGR}(1) is actually `if and only if'.
\item The category $\Ver_{p^\infty}$ is {\bf GR}.
\end{enumerate}

\end{conjecture}

\begin{remark}\label{Counter}
\begin{enumerate}
\item Conjecture~\ref{ConjNag}(1) for $p=2$ will be proved in \cite{CEO2}. For general~$p$ it will be proved for the subcategory $\Ver_{p^2}$ below in Proposition~\ref{Propp2}.
\item Conjecture~\ref{ConjNag}(1) would be false without the assumption of moderate growth. For example take $k=\mC$ and the semisimple pretannakian category $\cC=(\Rep GL)_t$ for $t\in\mC\backslash\mZ$, see \cite[\S 10]{Del07}. Then for the finitely generated algebra $A:=\Sym(X_t\oplus X_t^\vee)$, we find that $A^0$ is an $\mN$-graded algebra where the dimension of the degree $n$ part equals the number $p(n)$ of partitions of size $n$, which is not of polynomial growth.
\item It follows from Corollary~\ref{CorGRG} that every finite tannakian category (over a perfect field) is {\bf GR}. In characteristic zero this is {\em not} the case for finite {\em super}-tannakian categories. Indeed, consider the representation category of the 1-dimensional odd Lie superalgebra, with short exact sequence
$$0\to\bar{\unit}\to E\to\unit\to0,$$
with $\bar{\unit}$ the odd line, see also \cite[Remark~2]{Ve}. It is an interesting question whether arbitrary finite tensor categories in positive characteristic are {\bf GR}. An affirmative answer would in particular imply Conjecture~\ref{ConjNag}(2).
\end{enumerate}
\end{remark}

From now on we assume $\mathrm{char}(k)=p>0$, and we label the indecomposable tilting module of $SL_2$ with highest weight $i\in\mN$ by $T(i).$

\begin{prop}\label{Propp2}
The category $\Ver_{p^2}$ is {\bf GR}.
\end{prop}
\begin{proof}
It suffices to prove property {\bf (GR)} for $P\tto \unit$ with $P$ the projective cover of $\unit$. We will prove in fact that $\Sym^p P\tto\unit$ is split.
 For $p=2$ this is an easy direct computation; $\Sym^2P\simeq \unit^2$. 

Consider the (defining) symmetric monoidal functor 
\begin{equation}\label{TiltVer}
\Tilt SL_2\to \Ver_{p^2}.
\end{equation}
The indecomposable projective objects in~$\Ver_{p^2}$ are the images of $T(i)\in \Tilt SL_2$ for $p-1\le i<p^2-1$, where in particular $P$ is the image of $T(2p-2)$, see \cite{BEO}. Moreover, the functor~\eqref{TiltVer} is fully faithful on morphisms between the latter tilting modules, see \cite[Proposition~3.5 and Theorem~4.2]{BEO}.
The result for $p>2$ is therefore a direct consequence of Lemma~\ref{LemTilt}(2) below.
\end{proof}


\begin{lemma}\label{LemTop}
\begin{enumerate}
\item We have
$$\dim \Hom_{SL_2}(T(a),k)=\begin{cases}
1&\mbox{ if $a=2p^i-2$ for $i\in\mN$,}\\
0&\mbox{ otherwise.}
\end{cases}$$
\item We have
$$\dim \Hom_{SL_2}(T(a),L(2p-2))=\begin{cases}
1&\mbox{ if $a=2p$,}\\
1&\mbox{ if $a=2p^i-2p$ for $i\in\mZ_{>1}$,}\\
0&\mbox{ otherwise.}
\end{cases}$$
\end{enumerate}
\end{lemma}
\begin{proof}
For part (1), see \cite[Lemma~5.3.3]{Selecta} or \cite[Lemma~3.6]{BEO}. The methods for the latter lemma can also be applied to prove part (2).
\end{proof}

\begin{lemma}\label{LemTilt} Assume $p>2$.
Consider $\beta:T(2p-2)\tto k$ associated with the simple top $k=L(0)$ of $T(2p-2)$.
\begin{enumerate}
\item The induced morphism $\beta^p:\Sym^{p}T(2p-2)\tto k$ is split (in~$\Rep SL_2$).
\item Let the first line of the following diagram represent the canonical presentation of $\Sym^{p}T(2p-2)$: 
$$\xymatrix{
(\otimes^p T(2p-2))^{p-1}\ar[r]& \otimes^p T(2p-2)\ar[r] &\Sym^p T(2p-2)\ar[r]\ar@{->>}[d]^{\beta^p}&0\\
T(i)\ar[r]\ar@{-->}[u]& T(2p-2)\ar@{^{(}->}[u]& k.
}$$
There exists a direct summand $T(2p-2)$ of $\otimes^p T(2p-2)$ as indicated above, such that the composite from $T(2p-2)$ to $k$ is not zero, and for which every non-isomorphism $T(i)\to T(2p-2)$ with $p-1\le i< p^2-1$ leads to the above type of commutative diagram.
\end{enumerate}
\end{lemma}
\begin{proof}
The tilting module $T(2p-2)$ has length three, with simple top and socle $k=L(0)$ and $L(2p-2)$ in the middle. Note also that $T(2p-2)$ is a direct summand of the tilting module $\otimes^2 T(p-1)=\otimes^2L(p-1)$ and since $\Hom_{SL_2}(\otimes^2L(p-1),k)$ is one-dimensional, for part (1) it is clearly sufficient to prove that the $p$-th power of $\otimes^2L(p-1)\to k$ splits. Since $L(p-1)$ has dimension $p$ and is self-dual, the (dual of the) determinant gives us an $SL_2$-morphism $k\to \Sym^p(\otimes^2 L(p-1))$ from which to derive the splitting in (1).

Now we commence the proof of part (2). Firstly we claim that, for any $i>0$, we can choosen a decomposition of $\otimes^i T(2p-2)$ into a direct sum $T(2p-2)$ and a complement on which $\otimes^i\beta$ vanishes. This we can prove by iteration on $i$, so that it suffices to deal with the case $i=2$. This case follows immediately, since for $l>1$, we have $2p^l-2>2(2p-2)$, the claim follows from Lemma~\ref{LemTop}(1). We also used that $T(0)$ cannot appear as a direct summand by the structure of tensor ideals, see \cite[Theorem~5.3.1]{Selecta}.

By part (1) there thus exists a direct summand $T(2p-2)$ of $\otimes^p T(2p-2)$ for which the composite $T(2p-2)\to k$ is not zero and, by using the splitting in (1), we know there must be one for which the kernel of $T(2p-2)\to \Sym^p T(2p-2)$ is exactly the radical of $T(2p-2)$.

By Lemma~\ref{LemTop}, the only $T(i)$ (with $p-1\le i\le p^2-1$) which admit non-zero morphisms to $T(2p-2)$ are $T(2p-2)$ and $T(2p)$. Since a non-zero $T(2p)\to T(2p-2)$ has as image the radical, it suffices to consider this morphism and show it leads to a diagram as in the lemma. By what we know already, the composition $T(2p)\to\Sym^p T(2p-2)$ is zero, so the morphism from $T(2p)$ factors through the image $I$ of the left upper arrow in the diagram.

We can also observe that this upper left morphism in the diagram factors via the summand of the source, written up to braid isomorphism, $$(\wedge^2 T(2p-2) \otimes \otimes^{p-2}T(2p-2))^{p-1},$$
which has highest weight strictly lower than $p(2p-2)$. On the other hand, by Lemma~\ref{LemTop}(2) the lowest weight $i>2p$ for which $T(i)$ has $L(2p-2)$ in its top is $p(2p-2)$. Hence, the only summands in~$(\otimes^p T(2p-2))^{p-1}$ which are not sent to zero and which have a non-zero morphism to $L(2p-2)$ are those isomorphic to $T(2p)$. Therefore we can lift any morphism $T(2p)\to I$ to $(\otimes^p T(2p-2))^{p-1}$.
\end{proof}

\subsection{Maximal nilpotency}
Consider the following potential properties on $\cC$:
\begin{enumerate}
\item[{\bf(MN1)}] For every simple $L\not=\unit$ in~$\cC$, the algebra $\Sym L$ is finite.
\item[{\bf (MN2)}] For every non-split $\alpha:\unit\hookrightarrow X$ in~$\cC$, there exists $n$ for which $\alpha^n:\unit\to\Sym^nX$ is zero.
\end{enumerate}

If a tensor category satisfies both {\bf (MN1)} and {\bf (MN2)}, we say it is {\bf maximally nilpotent}, or simply {\bf MN}.

\begin{lemma} Condition {\bf (MN2)} is equivalent to:

{\rm If $\mathrm{char}(k)=0$:} $\cC$ is semisimple;

{\rm If $\mathrm{char}(k)=p>0$:} For every non-split $\alpha:\unit\hookrightarrow X$ in~$\cC$, there exists $j\in\mN$ for which $\alpha^n:\unit\to\Sym^nX$ is zero if and only if $n\ge p^j$.
\end{lemma}
\begin{proof}
The ideas are the same as in the proof of Lemma~\ref{LemGR}.
\end{proof}

\begin{prop}\label{Prop4Bas}
The following conditions on $\cC$ are equivalent:
\begin{enumerate}
\item[(a)] $\cC$ is {\bf MN}.
\item[(b)] For every $A\in \Alg\cC$, we have $R(A)\subset \Nil A$, with $R(A)$ defined in \ref{A0A0}.
\item[(c)] Every domain in~$\Alg\cC$ is an integral domain (in~$\Alg_k$).
\end{enumerate}
\end{prop}
\begin{proof}
First we show that (b) implies (a). When applying (b) to $A=\Sym L$, for a simple $L\not=\unit$, the observation $L\subset R(\Sym L)$ proves {\bf (MN1)}. Similarly, applying (b) to $A=\Sym X$, with $X$ as in {\bf (MN2)}, implies {\bf (MN2)}.

To show that (a) implies (b), we need to prove that for any $\cC\ni Y\subset \widetilde{A}\subset A$, we have $Y^n=0$ for some $n\in\mN$. We prove this by induction on the length of $Y$. If $Y$ has length 1, then depending on whether $Y\not=\unit$ or $Y=\unit$, {\bf (MN1)} or {\bf (MN2)} implies the required $Y^n=0$.

Now take an arbitrary $\cC\ni Y\subset \widetilde{A}\subset A$ and assume that for some $Y'\subset Y$ with $Y/Y'$ simple we know that $(Y')^n=0$ for some $n$. Since the ideal $ Y'A$ is nilpotent in~$A$ is now sufficient to prove that the image of $Y/Y'\to A/Y'A$ is nilpotent. Since that image is either zero or simple, we are back at the case from the previous paragraph.

Finally we show that (b) and (c) are equivalent. Clearly, (c) is equivalent with the condition $R(A)\subset \fp$ for every prime ideal $\fp$ in every $A\in\Alg \cC$. This is indeed equivalent to (b) by Lemma~\ref{LemNil}.
\end{proof}

\begin{conjecture}
The category $\Ver_{p^\infty}$ is {\bf MN}.
\end{conjecture}

\begin{example}
The categories $\Ver_p$ and $\Ver_4$ are {\bf MN}. It will be proved in \cite{CEO2} that this extends to $\Ver_{2^\infty}$ and $\Ver_9$.
\end{example}

We apply some results from \cite{CEO2} (that do not logically depend on the following theorem) to demonstrate further the close connection between the condition {\bf (MN)} and incompressible categories.

\begin{theorem}\label{MNIncomp}
Let $k$ be algebraically closed and $\cC$ an {\bf MN} pretannakian category over $k$.
\begin{enumerate}
\item $\cC$ is of moderate growth.
\item If $\mathrm{char}(k)=0$, then $\cC$ is $\Vecc$ or $\sVec$. In particular $\cC$ is incompressible.
\item If $\cC$ admits a tensor functor to a union of finite tensor categories, then $\cC$ is incompressible.
\item If $\mathrm{char}(k)=p>0$ and $\cC$ is Frobenius exact, then $\cC\subset\Ver_p$. In particular $\cC$ is incompressible.
\end{enumerate}
\end{theorem}
\begin{proof}
For (1), assume first that $\mathrm{char}(k)=0$. Then the claim follows from \cite[Proposition~0.5]{Del02}. In case $\mathrm{char}(k)=p>0$, take an arbitrary $X\in\cC$ and consider a short exact sequence
$$0\to Y\to X\to Z\to 0$$
such that $Y$ is the minimal subobject for which the quotient $Z$ only has simple constituents isomorphic to $\unit$. In particular $\Hom(Y,\unit)=0$. It follows from Proposition~\ref{Prop4Bas}(b) applied to $\Sym Y$ that $\Sym^nY=0$ for some $n$. Hence the $n$-th symmetriser acts trivially on $Y^{\otimes n}$ and it follows from \cite[Proposition~4.11(5) and (6)]{CEO} that the length of $Y^{\otimes n}$ can be bounded by $C^n$ for some $C\in \mathbb{R}$. If we let $\ell$ denote the length of $Z$, then it follows that the length of $X^{\otimes n}$ is bounded by $(C+\ell)^n$.

Next we observe, and henceforth use freely, that every tensor functor from $\cC$ to another pretannakian category is fully faithful, by \cite[Theorem~7.2.3(1)]{CEO2}.

Now we prove parts (2) and (4). By \cite{CEO, Del02}, there exists a tensor functor $F:\cC\to\Ver_k$. The target is semisimple, so fullness of $F$ makes $\cC$ a tensor subcategory of $\Ver_k$.

Finally, we prove part (3). If $\cC$ is finitely generated as a tensor category, then a tensor functor to a union of finite tensor categories must actually take values in a finite tensor category. A union of incompressible tensor categories is incompressible, so we can assume that there is a tensor functor $F:\cC\to\cD$ to a finite tensor category $\cD$. Without loss of generality we can assume it is surjective in the sense that every object in $\cD$ is a subquotient of an object in the image of $F$. Since projectives in $\cD$ are injective, it means that every projective $P$ in $\cD$ is a direct summand of an object $F(X)$. By fullness, we can produce a corresponding idempotent in $\End(X)$, so that there is actually $Y\in\cC$ with $F(Y)\simeq P$, and fullness of $F$ shows that $Y$ is projective. It now follows quickly (for instance via \cite[Corollary~3.1.6(2)]{CEO2}) that $F$ is actually an equivalence of categories, so the result follows from \cite[Theorem~7.2.3(2)]{CEO2}.
\end{proof}

\begin{remark}
By \cite[Conjecture~1.4]{BEO}, the assumption in (3) is expected to be satisfied for any pretannakian category $\cC$ is of moderate growth. In particular, for $\mathrm{char}(k)=2$, the combination of Theorem~\ref{MNIncomp} and \cite[Theorem~9.2.1]{CEO2} show that \cite[Conjecture~1.4]{BEO} would imply that a pretannakian category of moderate growth is incompressible if and only if it is {\bf MN}.
\end{remark}


\section{Localisation and faithfully flat morphisms}\label{SecLoc}
 Given a prime ideal $\p<A\in\Alg\cC$, in this section we investigate to which extent the idea of the classical localisation $A\to A_{\fp}$ extends.

\subsection{Classical localisation}

\subsubsection{}\label{DefLoc}For a $k$-algebra $R$, $f\in R$ and $M\in\Mod_{\cC}R$, we can define the localisation
\begin{equation}\label{LocM}M_f\;:=\;R_f\otimes_RM\;\simeq\; \varinjlim M\end{equation}
where the direct limit is taken along $f:M\to M$. Clearly $(M^0)_f\simeq (M_f)^0$.


By equation~\eqref{EndA}, we can apply this to $A\in\Alg\bC$ and $f\in A^0$. Then we have isomorphisms
$$A_f\;\simeq\; A\otimes_{k[x]}k[x,x^{-1}]\;\simeq\; A^0_f\otimes_{A^0} A$$
with $k[x]=\Sym\unit\to A$ the algebra morphism defined by $f$. In particular, $A_f$ has a canonical algebra structure and $A\to A_f$ is flat, since $$A_f\otimes_AM\;\simeq\; M_f\;\simeq\; \varinjlim M.$$

More generally, for a multiplicative subset $S\subset A^0$, we define the localisation 
\begin{equation}\label{eqAS}
A[S^{-1}]=\varinjlim A_f=A^0[S^{-1}]\otimes_{A^0} A.
\end{equation}
As usual, the direct limit runs over the quotient of $S$ by the equivalence relation coming from the quasi-order on $S$ where $s\le t$ if $s$ divides a power of $t$.

\begin{remark}\label{RemLocInv}
We could define localisation more generally with respect to morphisms $L\to A$ for invertible objects $L\in \cC$, rather than only morphisms $\unit\to A$.
\end{remark}

\subsection{Visibility of prime ideals}

\begin{definition}
We say that a prime ideal $\p\in\Spec A$ is {\bf visible} if it is the kernel of an algebra morphism $A\to S$ to a simple algebra $S$.
\end{definition}

\begin{conjecture}\label{ConjVis}
Assume that $k$ is algebraically closed and $\cC$ of moderate growth. Then every prime ideal in every algebra in $\Alg\cC$ is visible.
\end{conjecture}

\begin{example}\label{ExVis1}
\begin{enumerate}
\item Every maximal ideal is visible.
\item If $\cC$ is unipotent (all simple objects are isomorphic to $\unit$), then every prime ideal is visible. Indeed, for every algebra $A\in\Alg\cC$ such that $A^0$ is simple, clearly $A$ is simple as well. In particular we can apply Lemma~\ref{LemVis0} below.
\item Conjecture~\ref{ConjVis} is true in characteristic zero and for Frobenius exact categories, by Theorem~\ref{NPVis}(2) and the main results of \cite{Del02, CEO}.
\end{enumerate}
\end{example}

\begin{remark}\label{RemVis}
\begin{enumerate}
\item For an algebra morphism $A\to B$, the induced map $\Spec B\to \Spec A$ sends visible prime ideals to visible prime ideals.
\item One can easily extend the proof of Example~\ref{ExVis1}(2) to the case where $\cC$ is pointed (every simple object is invertible), see for instance Remark~\ref{RemLocInv}.
\end{enumerate}

\end{remark}

\begin{lemma}\label{LemVis0}
Assume that every domain~$D\in\Alg\cC$ with $D^0$ simple, is simple itself. Then every prime ideal in every algebra in~$\Alg\cC$ is visible.
\end{lemma}
\begin{proof}
It suffices to show that every domain can be embedded into a simple algebra. For a domain~$A$, it follows that $A^0$ is an integral domain. Consequently 
$$D:=\Frac A^0\otimes_{A^0}A=A[S^{-1}],$$
for $S=A^0\backslash\{0\}$, is a domain with $D^0=\Frac A^0$ and hence $D$ is simple.
\end{proof}

\begin{theorem}\label{NPVis}Assume that $k$ is algebraically closed. Every prime ideal of every algebra in~$\Alg\cC$ is visible whenever one of the following assumptions on $\cC$ is satisfied:
\begin{enumerate}
\item $\cC$ is {\bf MN}.
\item There exists a tensor functor $F:\cC\to\cD$ to a semisimple {\bf MN} pretannakian category~$\cD$ (morally $\cD=\Ver_{k}$).
\item There exists a tensor functor $F:\cC\to\cD$ to an {\bf MN} pretannakian category $\cD$ for which~$F_\ast$ is faithful and exact (for instance~$\cC$ is finite and $F$ is surjective), and $k$ is perfect.
\end{enumerate}
\end{theorem}

We start the proof with the following proposition.

\begin{prop}\label{BetterProp}
Consider a tensor functor $F:\cC\to\cD$, an algebra $A\in\Alg\cC$ and $\fp\in\Spec A$. There exists $P\in\Spec (FA)$ so that $\fp$ is the preimage of $F_\ast P$ under the canonical algebra morphism $A\to F_\ast FA$.
\end{prop}
\begin{proof}
The lax monoidal structure of $F_\ast$ allows us to consider $F_\ast$ also as a functor $\Alg\cD\to \Alg\cC$, and for any algebra morphism $FA\to A'$, the associated $A\to F_\ast A'$ is again an algebra morphism. Applying this to $A\to F_\ast FA$ allows us to associate to any ideal $J<FA$ an ideal
$$A>J_F\;:=\; \ker(A\to F_\ast FA\to F_\ast (F(A)/J)).$$

Using the unit $\Id\Rightarrow F_\ast F$ (which evaluates to monomorphisms by faithfulness of $F$) and counit $F_\ast F\Rightarrow\Id$, we find that for ideals $J,J'<FA$ and $I<A$,
$$F(J_F)\subset J,\quad (F I)_F=I,\quad \mbox{and}\quad J_FJ'_F \subset (JJ')_F.$$
 
 By the second property, the set of ideals $J<FA$ with $J_F=\fp$ is not empty.  
 Zorn's lemma then implies existence of an ideal $P<FA$ which is maximal under the condition $\fp=P_F$. If~$P$ were not prime, we would have ideals $J_1,J_2<FA$ not included in~$P$ with $J_1J_2\subset P$. By maximality of $P$, we find that neither of $(P+J_i)_F$ is in the prime $\fp$. Consequently, we find
$$\fp=P_F\;\supset\; ((P+J_1)(P+J_2))_F \;\supset\; (P+J_1)_F (P+J_2)_F \;\not\subset\; \fp,$$
a contradiction. Hence $P$ is prime indeed.
\end{proof}

\begin{proof}[Proof of Theorem~\ref{NPVis}]
Part (1) is immediate and essentially a special case of part (3). Part (2) is also a special case of part (3), if we replace by $\cD$ by a tensor subcategory whenever possible (so that $F$ becomes surjective), and if we observe that \cite[Theorem~5.4.1]{CEOP} eliminates the requirement for $k$ to be perfeect). Hence we only prove part (3).

 Essentially by Beck monadicity (see e.g. \cite[Lemma~6.2.1]{CEO2}), 
$$\Ind\cD\;\simeq\; \Mod_{\cC}S$$ 
 for $S:=F_\ast \unit$. In particular, the algebra $S$ is simple. By applying the same reasoning to the composition $\cC\to\cD\to \cD_K$ where the second functor is the extension of scalars along a field extension $K/k$ and $\cD_K$ is again a pretannakian category over $K$ by \cite[Corollary~5.2.4]{CEOP}, we observe that $K\otimes S:=F_\ast(K)$ is also simple.

 Now consider $\fp\in\Spec A$. Consider the prime ideal $P<FA$ from Proposition~\ref{BetterProp}. Since $FA/P$ is a domain, it is an integral domain by Proposition~\ref{Prop4Bas}, so we can embed it into its field of fractions $K$. We thus obtain an inclusion
 $$A/\fp\;\hookrightarrow\; F_\ast (FA/P)\;\hookrightarrow\; F_\ast (K)=K\otimes S,$$
 concluding the proof.
\end{proof}

Note that the condition that $\cO(G)$ be an integral domain is tautologically also a necessary condition for the following consequence.
\begin{corollary}
Let $G$ be an affine group scheme with $\cO(G)\in\Alg_k$ an integral domain and assume that $k$ is algebraically closed. For any $A\in\Alg\Rep_k G$, the prime ideals in $A$ are precisely the prime ideals in $\omega A$ which are $G$-stable.
\end{corollary}
\begin{proof}
The only thing which requires proof is that a prime ideal $\fp$ in $A$ is still prime when considered as an ideal of $\omega A$. By Proposition~\ref{BetterProp}, there exists a prime ideal $P<\omega A$ so that the co-action induces an injective algebra morphism
$$\omega (A/\fp)\;\hookrightarrow\; \cO(G)\otimes (\omega A)/P. $$
The right-hand side is the tensor product of two integral domains and therefore also an integral domain.
\end{proof}


%
%
%
%
%

\subsection{Faithfully flat morphisms}

\begin{prop}\label{LemFlat}
For a flat algebra morphism $f:A\to B$ in~$\Ind\bC$, the following are equivalent:
\begin{enumerate}
\item $B$ is faithfully flat over $A$;
\item every maximal ideal in~$A$ is in the image of $\Spec f$;
\item $\Spec f$ is a surjection when restricted to visible prime ideals (see Remark~\ref{RemVis}(1));
\item $B\otimes_A (A/\p)\not=0$ for all $\p\in\Spec A$.
\end{enumerate}
\end{prop}

We start the proof with the following lemmata.

\begin{lemma}\label{LemObv}For an algebra morphism $f:A\to B$ and $\p\in\Spec A$, the following are equivalent:
\begin{enumerate}
\item $B\otimes_A (A/\p)\not=0$;
\item $Bf(\p)\not=B$;
\item there is $\q\in\Spec A$ with $\p\subset \q$ and $\q$ in the image of $\Spec f$.
\end{enumerate}

\end{lemma}
\begin{proof}
We have an exact sequence
$$B\otimes_A \p\to B\to B\otimes_A(A/\p)\to 0,$$
where the image of the first morphism is $Bf(\p)$, showing that (1) and (2) are equivalent. 

Condition (2) is equivalent to $f(\p)\subset \fm$ (or $\fp\subset f^{-1}(\fm)$) for some maximal ideal $\fm<B$. This shows the equivalence of (2) and (3).
\end{proof}

\begin{lemma}\label{LemObv2}
For a flat algebra morphism $f:A\to B$ and $\p\in\Spec A$ a visible prime ideal, the following are equivalent:
\begin{enumerate}
\item the conditions in Lemma~\ref{LemObv} are satisfied;
\item $\p$ is in the image of $\Spec f$;
\item $\p$ is the image of a visible prime ideal under $\Spec f$.
\end{enumerate}
\end{lemma}\begin{proof}
It is obvious that (3) implies (2) and that (2) implies (1).

Now we show that (1) implies (3).
Let $A\to S$ be a morphism to a simple algebra with kernel $\p$. We have the pushout
$$\xymatrix{
A\ar[rr]^f\ar[d]&& B\ar[d]\\
S\ar[rr]&& B\otimes_AS.
}$$
We can observe that $B\otimes_AS$ is not zero since, by flatness of $f$, the kernel of the morphism from $B$ is precisely $Bf(\p)\simeq B\otimes_A\p$ (by assumption a proper ideal). We apply $\Spec$ to the diagram. Any maximal ideal $\m$ in~$B\otimes_AS$ goes to the zero ideal in~$S$ and hence to $\p$ in~$A$, showing $\p$ must also be the image under $\Spec(f)$ of some $\q<B$. As $\q$ is the image of a maximal ideal, it is visible, see Remark~\ref{RemVis}(1).
\end{proof}

\begin{proof} [Proof of Proposition~\ref{LemFlat}]
It is obvious that (1) implies (4). That (4) implies (3) follows from Lemma~\ref{LemObv2}. That (3) implies (2) is also obvious.

That (2) implies (1) follows from the combination of Lemmata~\ref{LemObv2} and~\ref{FaithMax}.
\end{proof}

\begin{prop}\label{PropLazard}
Consider a morphism $f:R\to S$ in~$\Alg_k$. If $f$ is flat (resp. faithfully flat), then $f$ is also flat (resp. faithfully flat) when considered as a morphism in~$\Alg\cC$.
\end{prop}
\begin{proof}
Assume that $f$ is flat. By the Govorov - Lazard Theorem, the $R$-module $S$ is a filtered colimit of finitely generated free $R$-modules. Since the tensor product in~$\Ind\cC$ is cocontinuous, the fact that $f$ remains flat in~$\Alg\cC$ thus follows from the observation that $M\otimes_R-$ is exact on $\Mod_{\cC}R$ for a free $R$-module $M\simeq R^\beta$ (for some cardinality $\beta$).

If additionally $f$ is faithful, then the conclusion follows since the characterisations in Proposition~\ref{LemFlat} are independent of the ambient tensor category.
\end{proof}

\begin{lemma}\label{AIMI}
Consider $A\in\Alg\cC$ with a nilpotent ideal $I<A$ and $M\in \Mod_{\cC}A$ such that $M/IM$ is flat as an $A/I$-module and $I\otimes_A M\to IM$ is a monomorphism. Then $M$ is flat as an $A$-module.
\end{lemma}
\begin{proof}
The condition on $I\otimes_A M\to IM$ can be reformulated as $\Tor_1^A(A/I,M)=0$. This in turn implies that the surjection, for any $A/I$-module $N$,
$$\Tor_1^{A}(N,M)\;\tto\; \Tor_1^{A/I}(N,M/IM)$$
is in fact an isomorphism. As by assumption the right-hand side is zero, $\Tor_1^A(N,M)=0$ for all $A/I$-modules $N$. Since $I$ is nilpotent, the claim then follows by iteration.
\end{proof}

\subsection{Maximal ideals}

\subsubsection{} We denote by $\LocAlg\cC$ the category of local algebras in~$\Alg\cC$ with local morphisms (the maximal ideal is sent into the maximal ideal). We have the forgetful functor
\begin{equation}
\label{DefU}
U:\LocAlg\cC\to\Alg\cC,
\end{equation}
which we will mostly omit from notation.

\subsubsection{}\label{DefAm}Let $\m$ be a maximal ideal in~$A\in\Alg\cC$. We denote by $A_{\m}\in\LocAlg\cC$ (when it exists) the representing object of the subfunctor
$$F_{\fm}\subset\Alg(A,U-):\;\LocAlg\cC\to\Set$$
such that $F_{\fm}(B)$ comprises the algebra morphisms $A\to B$ for which $\m$ is in the image of $\Spec B\to\Spec A$ (and so automatically the image of the maximal ideal in~$B$).

Then we can consider the following condition on $\cC$:
\begin{enumerate}
\item[{\bf (Loc)}] For every $A\in\Alg\cC$ and every maximal ideal $\m<A$, the morphism $A\to A_{\m}$ exists and is flat.
\end{enumerate}

We will see that the combination of {\bf (GR)} and {\bf (MN)} implies {\bf (Loc)}.

\begin{lemma}\label{LemLoc}
If {\bf (Loc)} is satisfied, then for every $A\in\Alg\cC$ the induced functor
$$\Mod_{\cC}A\;\to\; \prod_{\m} \Mod_{\cC}A_{\m},$$
with product taken over all maximal ideals $\m<A$, is exact and faithful. 
\end{lemma}
\begin{proof}
By assumption the functor is exact. That it is faithful then follows from Lemma~\ref{FaithMax} and the observation
$$A_{\fm}\otimes_A (A/\m)=A_{\fm}/\fm A_{\m}\not=0,$$
since $\m$ is sent into the maximal ideal under $A\to A_{\m}$.
\end{proof}

Now we demonstrate that {\bf (Loc)} is not an empty condition, since already the existence of $A_{\m}$ (without demanding flatness) is not guaranteed.

\begin{lemma}\label{LemNoRea}
Consider $A\in\Alg\cC$ with maximal ideal $\m<A$ such that $\hat{A}:=\varprojlim A/\m^n$ is such that the image of $\Spec \hat{A}\to\Spec A$ contains more than one maximal ideal (in particular $\hat{A}$ is not local). Then $A_{\m}$ does not exist.
\end{lemma}
\begin{proof}
Consider the inverse system $\{A/\m^n\,|\, n\in\mZ_{>0}\}$ in~$\LocAlg\cC$. Assume for a contradiction that $A_{\m}$ exists. By universality of $A_{\m}$ there exists a cone $\{A_{\m}\to A/\m^n\}_{n}$, such that the composites $A\to A_{\m}\to A/\m^n$ are just the quotient maps. In particular, we have algebra morphisms
$$A\;\to\;A_{\m}\;\to\; \hat{A},$$
which compose to the natural morphism $A\to\hat{A}$. However, considering the corresponding map
$$\Spec \hat{A}\;\to\;\Spec A_{\m}\;\to\;\Spec A,$$
yields a contradiction, since the only maximal ideal in the image of the right map is $\m$.
\end{proof}

\begin{example}\label{ExAm}
Set $k=\mC$ and $\cC=\Rep\mG_a$ and consider $A:=\Sym V$ with $V=M_2$, so $A=\mC[x,y]$, with $Dx=y$. Then $A_{\m}$ does not exist. Indeed, consider the maximal ideal $\m=\Sym^+V$. By Example~\ref{ExInverse}, we have 
$$\hat{A}:=\varprojlim \Sym V/\Sym^nV\simeq \mC[x][[y]],$$ with natural $\mG_a$-action.
Then the image of $\Spec \hat{A}\to\Spec A$ contains every maximal ideal of the form $(x-\mu,y)$ with $\mu\in \mC$. The conclusion follows from Lemma~\ref{LemNoRea}.
\end{example}

\subsection{The geometric spectrum}

Although our definition of $\Spec A$ is classical and logical, there is another potential definition which is perhaps more meaningful (in particular, it would be the expected underlying topological space of a geometric realisation of $A$ as a locally ringed space). We present it here but do not dwell on it for two reasons. Firstly, as we show in Example \ref{ExGSpec}, it need not always be defined. Secondly, we have no examples of tensor categories where it is defined but different from $\Spec A$.

\subsubsection{}\label{Defmulti} We use the terminology and notation from \cite{Diers}.
Assume that $U$ in \eqref{DefU} has a left multi-adjoint ({\it i.e.} the subcategory $\LocAlg\cC$ of $\Alg\cC$ is multireflexive), then we define the {\bf geometric spectrum} of $A\in\Alg\cC$ as
$$\GSpec A\;:=\; \Spec_UA.$$

More concretely, when $\GSpec A$ is defined, it labels a family of algebra morphisms $A\to A_\alpha$, with $A_\alpha$ local, so that
$$\Alg\cC(A,U-)\;\simeq\; \bigsqcup_{\alpha\in\GSpec A}\LocAlg\cC(A_\alpha,-).$$
There is a function
\begin{equation}\label{Gmap}
\GSpec A\;\to\;\Spec A
\end{equation} which sends $\alpha$ to the preimage of the maximal ideal of $A_\alpha$ in~$A$.

\begin{lemma}\label{LemGSpec}
Keep the assumption in \ref{Defmulti} on $\cC$ (so that $\GSpec$ is defined).
\begin{enumerate}
\item The image of \eqref{Gmap} is precisely the set of visible prime ideals.
\item The fibre in~$\GSpec A$ of every maximal ideal in~$\Spec A$ is a singleton.
\end{enumerate}
\end{lemma}
\begin{proof}
Part (1) is easy. For part (2), fix a maximal ideal $\m <A$. Consider one of the universal $A\to A_\alpha$, which is sent to $\fm$ under \eqref{Gmap}, and denote the maximal ideal in~$A_\alpha$ by $\m_\alpha$. The morphism $A\to A_\alpha/\fm_\alpha$ has kernel~$\m$, yielding a commutative square
$$\xymatrix{
&A_\alpha\ar[rd]&\\
A\ar[ru]\ar[rd]&&A_\alpha/\m_\alpha\\
&A/\m\ar[ru].&
}$$ Now $A\to A/\fm$ must factor through precisely one of the $A\to A_\beta$. It can only be $A_\alpha$, since otherwise $A\to A_\alpha/\m_\alpha$ would factor through 2 of the universal morphisms. However, $\alpha$ was arbitrary, so it must be unique.
\end{proof}

\begin{example}\label{ExGSpec}
Set $k=\mC$ and $\cC=\Rep\mG_a$, then $\LocAlg\cC\subset \Alg\cC$ is not multireflexive. Indeed, this follows from Lemma~\ref{LemGSpec}(2) and Example~\ref{ExAm}.
\end{example}


\section{Comparing spectra}\label{SecComp}

\subsection{Some natural maps}
Consider $A\in\Alg\cC$.

\subsubsection{}\label{diagSpec}
We consider the diagram of algebra morphisms (see~\ref{A0A0}) and its image under $\Spec$:
$$
\xymatrix{
&A\ar@{->>}[rd]&         &&&\Spec A\ar[ld]_{\rho_A}\\
A^0\ar@{^{(}->}[ru]\ar@{->>}[rd]&&A_0  &&\Spec A^0&&\Spec A_0\ar[ld]\ar@{}[lu]|-*[@]{{ \subset}}\\
&A^0_0\ar@{^{(}->}[ru]&    &&&\Spec A^0_0,\ar@{}[lu]|-*[@]{\subset}
}
$$
where $A^0_0$ is just defined as the image of the composite $A^0\hookrightarrow A\tto A_0$.
Alternatively we can regard $\{\rho_A\}$ as the natural transformation 
$$\rho:\;\Spec \Rightarrow \Spec\circ H^0,\quad \rho_A:\;\Spec A\to \Spec A^0,$$
of functors $\Alg\cC\to\Top$.

%
%
%
%

\subsubsection{}\label{3rho}We will be interested in situations where the maps are homeomorphisms. It will be convenient to consider this condition as the combination of the following three conditions
\begin{enumerate}
\item $\rho_A$ is injective;
\item $\rho_A$ is surjective;
\item the topology on $\Spec A$ is the pullback topology under $\Spec A\to\Spec A^0$.
\end{enumerate}
We start with some technical observations, implying in particular that condition (3) actually implies condition (1).

\begin{lemma}\label{LemSur}
The following conditions on $\cC$ are equivalent:
\begin{enumerate}
\item For every $A\in\Alg\cC$, the map $\rho_A$ is surjective.
\item For every $A\in\Alg\cC$ and proper ideal $I<A^0$, we have $AI\not=A$.

\item For every $A\in\Alg\cC$, the image of $\rho_A$ contains all maximal ideals in~$ A^0$.
\end{enumerate}
\end{lemma}
\begin{proof}
First we prove that (2) implies (3). Let $M$ be a maximal ideal in~$A^0$. Let $\fm$ be a maximal ideal which contains $AM\not=A$. Then $\fm$ is prime and clearly $\fm^0=M$.

To see that (3) implies (2), we can observe that $I$ is contained in a maximal ideal $M<A^0$. Under assumption (3), there is a proper (prime) ideal $\fp$ in~$A$ with $\fp^0=M$,
implying $AI\subset AM \subset \fp$.

Obviously (1) implies (3). We prove that (3) implies (1).
Consider an arbitrary $A\in\Alg\cC$ and $P\in\Spec A^0$. Set $B=A[S^{-1}]$ for $S=A^0\backslash P$. Then $B^0$ is the localisation of $A^0$ at the prime ideal $P$, so it has unique maximal ideal $M$, by assumption in the image of $\rho_B$. The commutative diagram
$$\xymatrix{
\Spec B\ar[rr]^{\rho_B}\ar[d]&& \Spec B^0\ar[d]\\
\Spec A\ar[rr]^{\rho_A}&& \Spec A^0
}$$
shows that $P$, which is in the image of the upper path, must be in the image of $\rho_A$.
\end{proof}

\begin{example}\label{exss}
 If $\cC$ is semisimple, it follows from Lemma~\ref{LemSur} that $\rho_A$ is always surjective.
\end{example}

\begin{lemma}\label{LemPull}
For $A\in\Alg\cC$, the following conditions are equivalent:
\begin{enumerate}
\item The topology on $\Spec A$ is the pullback of the topology on $\Spec A^0$;
\item $V(I)=V(AI^0)$ for every ideal $I<A$.
\item $\rho_A:\Spec A\to \Spec A^0$ is the inclusion of a subspace.
\end{enumerate}
\end{lemma}
\begin{proof}
If $V(I)=V(AI^0)$, then $V(I)=\rho_A^{-1}(V(I^0))$, hence (2) implies (1).

To prove that (1) implies (2), it suffices to assume that $I$ is radical, by Corollary~\ref{CorNil}(3). If the topology on $\Spec A$ is the pullback topology, then
for every radical ideal $I<A$, there exists an ideal $J<A^0$ with $V(I)=\rho_A^{-1}(V(J))$. Clearly $V(I)=V(AJ)$. But now
$$I^0=\cap_{P\in V(I)}P^0 \supset J$$
from which $V(AI^0)\subset V(AJ)$. As we always have $V(AI^0)\supset V(I)$, it follows that $V(AI^0)=V(I)$.

Clearly (3) implies (1), so we conclude by proving that (1) + (2) implies (3). In fact, we can just show that (2) implies that $\rho_A$ is injective. For this, observe that (2) implies that for prime ideals $\fp,\fq <A$ we have $\fp\subset\fq$ if and only $\fp^0\subset\fq^0$.
\end{proof}

\subsection{Tannakian examples}

We will show that, for tannakian categories, in general all three conditions in \ref{3rho} fail, although $\rho$ is an isomorphism for infinitesimal group schemes.

\begin{example}
Set $k=\mC$ and $\cC=\Rep \mG_a$ (see~\ref{Ga}). Then $\rho_A$ is in general neither surjective nor injective:
\begin{enumerate}
\item Consider $B=k[x,y,u,v]$ with $Dx=u$ and $Dy=v$. 
Set $A:=B/(uy-vx-1)$. Let $I$ be the ideal of $A^0$ generated by $u$ and $v$, then clearly $AI=A$. On the other hand, we can calculate directly that $A^0=k[u,v]$, so $I\not=A^0$. By Lemma~\ref{LemSur}, $\rho$ is not always surjective. In particular, we can check directly that for $A$ as above, so $\Spec A^0=\mA^2$, we have $\im \rho_A=\mA^2\backslash\{0\}$.
\item
Set $A=\Sym M_2$. More concretely, $A=k[x,y]$ with $Dx=y$, so $A^0=k[y]$. Then, for each $\lambda\in\mC$, we have the maximal $(\mG_a$-stable) ideal $(x-\lambda,y)$. Every such ideal is sent to the ideal $(y)$ by $\rho_A$. So $\rho_A$ is not injective, but easily seen to be surjective.
\end{enumerate}
\end{example}

\begin{example}
Set $\cC=\Rep \mG_m$. Then $\rho$ is surjective (by Example~\ref{exss}), but not injective:

 Let $V$ be a 1 dimensional simple representation in degree 1 and set $A=\Sym V=k[x]$, so $A^0=k$. Then, the maximal $(\mG_a$-stable) ideal $(x)$ and the zero (prime) ideal are both sent to the zero ideal by $\rho_A$.
\end{example}

\begin{example}
Assume that $k$ is algebraically closed and set $\cC=\Rep SL_2$. For $A=\Sym V$ with $V$ the tautological representation, $\rho_A$ is not injective and does not define the topology on $\Spec A$. Indeed, we have $A^0=\unit$, since $\Sym^nV$ is the costandard module $\nabla(n)$. Moreover $\Spec A$ contains precisely two elements, the zero ideal and the augmentation ideal. The latter follows easily from the observation that, if char$(k)=0$, then $\Sym^nV$ is always simple, whereas if char$(k)=p>0$, then the Steinberg modules $\Sym^{p^n-1}V$ are simple.
\end{example}

\begin{prop}\label{PropG} Let $G$ be an infinitesimal group scheme and $A\in\Alg\Rep G$.
The two maps
$$\Spec A\;\xrightarrow{\rho_A}\; \Spec A^0 \;\leftarrow\; \Spec \omega A$$
are homeomorphisms.
\end{prop}
\begin{proof}
The crucial observation is the following. Let $p>0$ denote the characteristic of $k$, with power $q$ as in \ref{DefInf}. For each $a\in A$ (more precisely $a\in \omega A$), we have $a^q\in A^0$. 

In particular, we can combine the inclusion $A^0\subset \omega A$ with an algebra morphism (in~$\Alg_k$) to get a sequence
$$A^0\;\hookrightarrow\; \omega A\;\to\; A^0\;\hookrightarrow \omega A$$
in which every composition of two consecutive arrows is the Frobenius homomorphism $a\mapsto a^q$. As the latter yields the identity map on spectra, it follows that the right map is a homeomorphism.

That $\rho_A$ is surjective can be seen directly, or by observing that Proposition~\ref{BetterProp} yields a map $\Spec \omega A\to \Spec A$ turning the above two maps into a commutative diagram.

We thus only need to show that $\rho_A$ is the inclusion of a subspace. Observe therefore that for a radical ideal $I<A$, the following are equivalent conditions on an $\ell$-dimensional subrepresentation $V\subset A$ with basis $\{v_1,v_2,\cdots,v_{\ell}\}$
\begin{enumerate}
\item[(i)] $V\subset I$;
\item[(ii)] $v_i^q\in I^0$ for $1\le i\le \ell$;
\item[(iii)] $V^{\ell q}\subset I$.
\end{enumerate}
Indeed, that (i) implies (ii) and (ii) implies (iii) follows just by using the fact that $I$ is an ideal. That (iii) implies (i) follows since $I$ is radical. In particular, it follows that $I\subset J$ is equivalent to $I^0\subset J^0$ for radical ideals $I,J$.
\end{proof}

\begin{lemma}
Let $\Gamma$ be an abelian torsion group, then for $\Vecc_{\Gamma}\simeq\Rep \Gamma^\vee$, the natural transformation $\rho$ is an isomorphism.
\end{lemma} 
\begin{proof}
That $\rho$ is surjective follows from Example~\ref{exss}. That $\rho$ is the inclusion of a subspace follows from the observation that $J\mapsto J^0$ restricts to an isomorphism on the poset of radical ideals in~$A$, since for any radical ideal $J$ of an algebra in~$\Vecc_\Gamma$ and $g\in \Gamma$, we have $f\in J_g$ if and only if $f^n\in J_0$, for $n$ the order of $g$.
\end{proof}

\subsection{Reinterpreting the convenient assumptions}

\begin{theorem}\label{Thm4Spec} 
\begin{enumerate}
\item $\cC$ is {\bf MN} if and only if $\Spec A_0\subset\Spec A$ is an equality for every $A\in\Alg\cC$.
\item If $\cC$ is {\bf MN}, then $\Spec A_0^0\subset\Spec A^0$ is an equality for every $A\in\Alg\cC$.
\item If $\Spec A_0^0\subset\Spec A^0$ is an equality for every $A\in\Alg\cC$, then {\bf(MN2)} is satisfied.
\item The following conditions are equivalent:
\begin{enumerate}
\item The map $\Spec A_0\to \Spec A^0$ is injective.
\item The map $\Spec A_0\to \Spec A_0^0$ is a homeomorphism.
\item $\cC$ is {\bf GR}.
\end{enumerate}
\end{enumerate}
\end{theorem}
\begin{proof}
Part (1) and (2) are immediate applications of Proposition~\ref{Prop4Bas}.

For part (3), we observe first that $\Spec A_0^0\subset\Spec A^0$ being an equality is equivalent to the kernel of $A^0\to A_0$ being a nil ideal.
Now assume that $\Spec A_0^0\subset\Spec A^0$ is an equality for $A=\Sym X$ with some non-split $\unit\hookrightarrow X$. It follows that $a\in A^0$, corresponding to $\unit\hookrightarrow X\subset A$, is nilpotent, which implies {\bf (MN2)}.

Now we prove part (4). First show that (c) implies (b), so we assume that $\cC$ is {\bf GR}. If $\mathrm{char}(k)=0$ then $A^0_0=A_0$, see Theorem~\ref{ThmGR}(2), so we assume $p=\mathrm{char}(k)>0$.
 Theorem~\ref{ThmGR}(2) then shows that for any $A\in\Alg\cC$ and $a\in A_0$, there exists  a power $q$ of $p$ such that $a^q\in A^0_0\subset A_0$. It follows that for radical ideals $I,J$ in~$A_0$, we have $(I\cap A^0_0)\subset (J\cap A^0_0)$ if and only if $I\subset J$. Hence $\Spec A_0\to \Spec A_0^0$ is the inclusion of a subspace. Moreover, for a prime ideal $P$ of $A_0^0$, denote by $P'\subset A_0$ the subset of all $a\in A_0$ for which $a^q\in P$ for some power $q$ of $p$. Clearly $P'$ is a prime ideal in~$A_0$ and $P'\cap A^0_0=P$. 
It follows that (c) implies (b). Obliviously (b) implies (a).

To prove that (4)(a) implies 4(c), assume that {\bf (GR)} is not satisfied. That means we can construct $X\in\cC$ with simple top $\unit$ for which $\Sym^nX\tto \unit$ has no section for all $n\in \mN$. This implies that, for $A:=\Sym X$, we have $A^0_0=k$. This implies in turn that $\Spec A_0\to \Spec A^0$ factors as
$$\mA^1=\Spec A_0 \;\to\; \ast\;\to\; \Spec A^0$$
with $\ast$ the singleton. Hence the map is not injective, concluding the proof.
\end{proof}

\begin{corollary}\label{CorNPGR}
Assume that $\cC$ is {\bf MN}. Then the following are equivalent:
\begin{enumerate}
\item[(a)] $\cC$ is {\bf GR};
\item[(b)] $\rho_A$ is always a homeomorphism;
\item[(c)] $\rho_A$ is always an injection.
\end{enumerate}
\end{corollary}
\begin{proof}
If $\cC$ is {\bf MN} and {\bf GR}, then by Theorem~\ref{Thm4Spec} three of the four maps in the diagram in~\ref{diagSpec} are homeomorphisms, forcing the fourth one, $\rho_A$, to be a homeomorphism too. Hence (a) implies (b). Obviously (b) implies (c).

Finally, if $\rho_A$ is an injection, then so is $\Spec A_0\to \Spec A^0$. Hence Theorem~\ref{Thm4Spec}(4) shows that (c) implies (a).
\end{proof}

\subsection{Connection with localisation}

We start with some technical results, mainly for future use.

\begin{lemma}\label{LemfInv}
For $A\in\Alg\cC$ and $f\in A^0$, the following are equivalent:
\begin{enumerate}
\item No proper ideal of $A$ contains $f$.
\item As an element of $A^0$, $f$ is invertible.
\end{enumerate}
\end{lemma}
\begin{proof}
Obviously (2) implies (1). To prove the converse, we can observe that (1) implies that the ideal generated by $f$ is $A$. Consequently, there must exist $\cC\ni X\subset A$ such that $f(X)$ contains $\im \eta$. Let $Y\subset X$ denote the kernel of $f:X\tto f(X)$. Then 
$$0=Xf(Y)=f(XY)=Yf(X)\supset Y,$$
so $Y=0$ and there must be some $g:\unit\hookrightarrow X$, so in particular $g\in A^0$, such that $fg=1$.
\end{proof}

\begin{corollary}\label{CorLoc}
If $A\in\Alg\bC$ is local with maximal ideal $\fm$, then $A^0$ is local with maximal ideal $\fm^0$. 
\end{corollary}
\begin{proof}
Lemma~\ref{LemfInv} implies that, for $f\in A^0$, either $f\in \fm^0$ or $f$ is invertible. In particular $\fm^0<A^0$ is the unique maximal ideal.
\end{proof}

\begin{corollary}\label{CorLoc2}
Assume that the topology on $\Spec A$ is the pullback of the topology on $\Spec A^0$. The following conditions are equivalent on $A\in\Alg\cC$:
\begin{enumerate}
\item $A$ is local;
\item $A^0$ is local and $AN\not=A$ for the maximal ideal $N\subset A^0$.
\end{enumerate}
\end{corollary}
\begin{proof}
That (1) implies (2) is always true, see Corollary~\ref{CorLoc}.

Under assumption (2), it follows that $N\subset A^0$ is of the form $\fm^0$ for a maximal ideal $\fm\subset A$. For any other prime ideal $\fp\in \Spec A$, the observation $\fp^0\subset N=\fm^0$ 
implies $\fm\in V(A\fp^0)$. However, Lemma~\ref{LemPull} implies $V(A\fp^0)=V(\fp)$, implying that $\fm$ is the unique maximal ideal.
\end{proof}

\begin{prop}\label{PropLoc}
Assume that $\rho_A$ is the inclusion of a subspace for every $A\in\Alg\cC$, for instance~$\cC$ is {\bf GR} and {\bf MN}.
\begin{enumerate}
\item For every $A\in\Alg\cC$, the function $\GSpec A\to\Spec A$ from \eqref{Gmap} is a bijection. Concretely, for each $\p\in\Spec A$, there exists $A\to A_{\fp}$ in~$\Alg\cC$ which is the universal morphism to a local algebra for which the preimage in~$A$ of the maximal ideal is $\fp$. 
\item The morphism $A\to A_{\fp}$ from (1) is flat and an epimorphism.
\item Property {\bf (Loc)} is satisfied in $\cC$.
\end{enumerate}
\end{prop}
\begin{proof}
Part (3) is an obvious consequence of parts (1) and (2).

For part (1) and (2), consider a prime ideal $\fp<A$. It suffices to show that $A\to A[S^{-1}]$, for $S=A^0\backslash \fp^0$, as defined in \ref{DefLoc} is local; that the preimage in~$A$ of the maximal ideal is $\fp$; and that every such morphism from $A$ to a local algebra factors uniquely via $A[S^{-1}]$. Indeed, by construction $A\to A[S^{-1}]$ is flat and an epimorphism.

That $A[S^{-1}]$ is local follows from Corollary~\ref{CorLoc2}. Indeed that $\fp^0A[S^{-1}]\not=A[S^{-1}]$ follows from 
$$(\fp^0A[S^{-1}])^0\;=\; (\fp^0A)^0[S^{-1}]\;=\; \fp^0[S^{-1}].$$

Now consider a morphism $A\to B$ in~$\Alg\cC$ where $B$ is local with maximal ideal $\fm$ such that its preimage is $\p$. It follows from Lemma~\ref{LemfInv} that the image of every $f\in S$ in~$B^0$ is invertible. Consequently, the morphism factors via $A[S^{-1}]$. Uniqueness is obvious, which concludes the proof.
\end{proof}

We also record the following observation for future use.
\begin{lemma}\label{CorLocMor}
Assume that $\rho_A$ is the inclusion of a subspace for every $A\in\Alg\cC$.
Consider an algebra morphism $\phi:A\to B$ in $\Alg\cC$ for which
\begin{enumerate}
\item $A^0$ and $B$ (so also $B^0$, see \ref{CorLoc2}) are local;
\item $\phi^0:A^0\to B^0$ is local. 
\end{enumerate} 
Then $A$ and $\phi$ are local.
\end{lemma}
\begin{proof}
Consider the maximal ideal $N<A^0$. By Corollary~\ref{CorLoc2}, to show that $A$ is local it suffices to show that $AN\not=A$. If $AN=A$, it follows that $B\phi(N)=B$. However, since $\phi^0$ and $B$ are local, Corollary~\ref{CorLoc2} yields a contradiction.

Now we prove that $\phi$ is also local. We need to show that (for the maximal ideals $\fm_A<A$, $\fm_B<B$), we have $\fm_A\subset \phi^{-1}(\fm_B)$. By Lemma~\ref{LemPull}, the latter is equivalent with $\fm^0_A\subset \phi^{-1}(\fm_B)$. This is clearly satisfied since already $\fm^0_A\subset \phi^{-1}(\fm_B^0)$.
\end{proof}


\section{Some hereditary properties}\label{SecHer}

Denote a hypothetical potential property of a pretannakian category by {\bf (P)}. We say that~{\bf (P)} is {\bf hereditary} if for every tensor functor $F:\cC_1\to \cC_2$ to a pretannakian category~$\cC_2$ which satisfies {\bf (P)}, it follows that $\cC_1$ satisfies {\bf (P)} as well. An obvious example is the moderate growth property and Frobenius exactness (see for instance \cite{CEO}). Obvious examples of properties which are {\em not} hereditary are {\bf (GR)}, {\bf (MN)} and~{\bf (Loc)}.

\subsection{The Nakayama property}

The following potential property of a pretannakian category $\cC$ is clearly hereditary.

\begin{enumerate}
\item[{\bf (Nak)}] For every algebra $A$, ideal $I<A$ and finitely generated $A$-module $N$ with $IN=N$, we have $$\Ann_A(N)+I=A.$$
\end{enumerate}

If this `Nakayama property' is satisfied, then Nakayama's lemma holds:

\begin{theorem}\label{ThmNak}
\begin{enumerate}
\item If {\bf (Nak)} is satisfied in~$\cC$, then the following equivalent properties are satisfied:
\begin{enumerate}
\item For every $A\in\Alg\cC$, if $JN=N$ for finitely generated $A$-module $N$ and $J$ the intersection of all maximal ideals in~$A$, then $N=0$.
\item For every $A\in\Alg\cC$, if $N_1+JN=N$ for $A$-modules $N_1\subset N$, with $N$ finitely generated, and $J$ the intersection of all maximal ideals in~$A$, then $N_1=N$.
\end{enumerate}
\item If $\cC$ satisfies {\bf (Loc)}, then the equivalent conditions in (1) imply {\bf (Nak)}.
\end{enumerate}
\end{theorem}
\begin{proof}
We start by observing that (1)(a) and (b) are equivalent. Clearly (a) is a special case of (b). On the other hand, (b) follows from applying (a) to $N/N_1$.

That {\bf (Nak)} implies (1)(a) follows from the observation that for an arbitrary ideal $K<A$ the property $J+K=A$ implies $K=A$.

Now we prove part (2). Consider $A,N,I$ as in {\bf (Nak)}. If {\bf (Loc)} holds, then for every maximal ideal $\m<A$, we have an equality of ideals
$$\Ann_{A_{\m}}(N_{\m})\;=\; \Ann_A(N)A_{\m}\;<\; A_{\m}.$$ Indeed, by taking a generator $\cC\ni X\subset N$, this follows from applying the exact functor $A_{\m}\otimes_A-$ to the exact sequence of $A$-modules
$$0\;\to\; \Ann_A(N)\;\to\; A\;\to\; X^\vee\otimes N.$$
This means that by Lemma~\ref{LemLoc}, it suffices to show that
$$\Ann_{A_{\fm}}(N_{\m})+IA_{\m}\;=\; A_{\m},$$
for every maximal ideal $\m<A$. But since $IA_{\m}N_{\m}=N_{\m}$, there are only two options. Either~$IA_{\m}$ is contained in the maximal ideal of $A_{\m}$, so that assumption (1)(a) implies $\Ann_{A_{\m}}(N_{\m})=A_{\m}$, or $IA_{\fm}=A_{\m}$, concluding the proof.
\end{proof}

\begin{remark}
It is not obvious that Nakayama's lemma itself is a hereditary property.
\end{remark}

\begin{lemma}\label{Lem4Nak}
Consider $A\in\Alg\cC$, an ideal $I<A$ and an $A$-module $N$ with $IN=N$. Then $N=0$ whenever
\begin{enumerate}
\item $I$ is a nilpotent ideal; or
\item $I$ is a nil ideal and $N$ is finitely generated.
\end{enumerate}
\end{lemma}
\begin{proof}
Part (1) is obvious, as by iteration, we have $I^nN=N$ for all $n\in\mN$. Now assume that the $A$-module $N$ is generated by $\cC\ni X\subset N$. In particular $IX=N$. By writing $I=\varinjlim_\alpha I_\alpha$ for $\cC\ni I_\alpha \subset I$ and compactness of $X$, we find that $I_\alpha X\subset N$ must contain~$X$ for some $\alpha$. In particular, we have $JN=N$ for the nilpotent ideal $J:=AI_\alpha$, so we can reduce to part~(1).
\end{proof}

\begin{corollary}\label{NPNak}
If $\cC$ is {\bf MN}, it satisfies Nakayma's lemma (the properties in \ref{ThmNak}(1)).
\end{corollary}
\begin{proof}
Firstly, we observe that the conditions are always satisfied for $A\in \Alg_k\subset \Alg\cC$, by the classical Nakayama's lemma. Indeed, this follows using a (non-monoidal) faithful exact functor $\cC\to\Vecc$.

For $A,J$ as in \ref{ThmNak}(1)(a), we can take the $A_0$-module $N':=N/R(A)N$. Then $J':=J/R(A)$ is the Jacobson radical of $A_0$. Furthermore, $J'N'=N'$. By the above paragraph, we find $N'=0$, or equivalently $N=R(A)N$. 

Since $R(A)$ is nil (Proposition~\ref{Prop4Bas}), Lemma~\ref{Lem4Nak}(2) shows that $N=0$.
\end{proof}

\begin{theorem}
If $\cC$ is {\bf MN} and {\bf GR}, or more generally admits a tensor functor to a pretannakian category which satisfies those properties, then it satisfies {\bf (Nak)}.
\end{theorem}
\begin{proof}
This follows from the combination of Corollary~\ref{NPNak}, Theorem~\ref{ThmNak}(2) and Proposition~\ref{PropLoc}(3).
\end{proof}

\subsection{The free property}

The following potential property of a pretannakian category $\cC$ is clearly hereditary.

\begin{enumerate}
\item[{\bf (Free)}] For every algebra $A$ and object $X\in \cC$, every epimorphic $A$-module morphism 
$$A\otimes X\;\tto\; A\otimes X$$
is an automorphism.
\end{enumerate}

\begin{remark}
One can quickly prove that an epimorphism $A\tto A$ is automatically an automorphism.
\end{remark}

\begin{theorem}\label{ThmFree}
If $\cC$ is {\bf MN} and {\bf GR}, or more generally admits a tensor functor to such a pretannakian category, then it satisfies {\bf (Free)}.
\end{theorem}
\begin{proof}
First we observe that (for arbitrary $\cC$) the claim is true for $A=R\in \Alg_k\subset\Alg\cC$ Indeed, consider an exact faithful functor $\Ind\cC\to \Vecc^\infty$. It might not be monoidal, but since $R\otimes X$ can be interpreted as coproduct of $\dim_kR$ copies of $X$ this allows us to interpret $R\otimes X$ as an ordinary free $R$-module of finite rank. The conclusion is then standard and follows for instance by localising at every maximal ideal.

For the general case we consider the algebra morphism
$$\End_A(A\otimes X)\;\xrightarrow{\zeta}\; \End_{A_0}(A_0\otimes X)\,\simeq\, A_0\otimes_k\End(X),$$
where $\zeta$ is the composite of
$$\End_A(A\otimes X)\;\simeq\; \Hom(X,A\otimes X)\;\to\; \Hom(X,A_0\otimes X)\;\simeq\;\End_{A_0}(A_0\otimes X),$$
where the middle morphism is induced from $A\tto A_0$.

It is straightforward to see that $\zeta$ sends epimorphisms to epimorphisms. Moreover, for every $\theta$ in the right-hand side, $\theta^n$ is in the image of $\zeta$ for some $n\in\mZ_{>0}$. Indeed, it suffices to consider the subalgebra $A^0\otimes \End(X)$ of the left-hand side and use Theorem~\ref{ThmGR}(2). Finally, if a morphism $\psi$ is in the kernel of $\zeta$, then it must be nilpotent. Indeed, if $\psi$ is in the kernel of $\zeta$, the associated $X\to A\otimes X$ lands in~$R(A)\otimes X$. Since $R(A)$ is a nil ideal (see Proposition~\ref{Prop4Bas}) and $X$ is compact, this morphism thus factors through $X\to I\otimes X$ for a nilpotent $I\subset A$.

Now consider an epimorphism $\phi:A\otimes X\tto A\otimes X$. By the above paragraph $\zeta(\phi)$ is an epimorphism and thus has an inverse $\theta$ by the first paragraph. By the previous paragraph, $\theta^n =\zeta(\psi)$ for some $n>0$. Hence $\psi\phi^n$ is sent to $1$ by $\zeta$. Again by the previous paragraph,  $1-\psi\phi^n$ is nilpotent, so $\psi \phi^n$ is an isomorphism, which indeed forces $\phi$ to be a monomorphism.
\end{proof}

\subsection{The Hopf property}
The following potential property of a pretannakian category $\cC$ is hereditary by Remark~\ref{RemFaith}.
\begin{enumerate}
\item[{\bf (Hopf)}] For every commutative Hopf algebra $A$ in $\Ind\cC$ and Hopf subalgebra $B\subset A$, the algebra morphism $B\to A$ is faithfully flat.
\end{enumerate}

The following is a generalisation of \cite[Theorem~5.1 and Corollary~5.1]{Zu}, which proves that $\sVec$ satisfies {\bf (Hopf)}, with similar proof.
\begin{theorem}\label{ThmHopfFP}
Assume that $k$ is algebraically closed. If $\cC$ is {\bf MN} and {\bf GR}, or more generally admits a tensor functor to such a pretannakian category, then it satisfies {\bf (Hopf)}.
\end{theorem}
\begin{proof}
We will also use the dual language of affine group schemes from Section~\ref{SecGrp} below.

Consider a commutative Hopf algebra $A$ and Hopf subalgebra $B\subset A$.
It follows from the definitions that $A_0$ (and $B_0$) inherits the structure of a Hopf algebra over $k$ and the induced morphism
$B_0\to A_0$ is one of Hopf algebras.
By Proposition~\ref{Prop4Bas}(b), its kernel is a nil ideal, so $B_0$ has the same spectrum as its image in~$A_0$, which is a Hopf subalgebra of $A_0$. Since the latter inclusion is faithfully flat by the classical theory, it follows from Theorem~\ref{Thm4Spec}(1) that 
$$\Spec A\to \Spec B$$
 is surjective. By Proposition~\ref{LemFlat}, it is thus sufficient to prove that the inclusion is flat.

By Lemma~\ref{LemLoc} and Proposition~\ref{PropLoc}(3), it is sufficient to show that $B\to A\to A_{\m}$ is flat for every maximal ideal $\m<A$. 
By Corollary~\ref{CorProAlg} below, it is sufficient to consider $A$ finitely generated, so that (by Proposition~\ref{Prop4Bas}) maximal ideals are in bijection with $G(\unit)$ for $G$ the affine group scheme corresponding to $A$. It thus follows easily that we only need to prove $B\to A_{\fm}$ is flat for one maximal ideal, which we choose to be the ideal $A^+:=\ker \epsilon$. By construction $A_{A^+}$ is flat over $A_{B^+}:=A\otimes_BB_{B^+}$, so we will instead prove that $A_{B^+}$ is flat over $B$.

Consider the algebra morphism
\begin{equation}\label{EqAB}
A\otimes_BA\;\to\; A\otimes A/AB^+
\end{equation}
which on the first factor $A$ is given by the obvious $A\to A\otimes \unit\to A\otimes A/B^+A$. On the second factor, it is given by $A\xrightarrow{\Delta} A\otimes A\tto A\otimes A/B^+A$. It is easily verified that~\eqref{EqAB} is in fact an isomorphism.

Using $\epsilon\circ \eta=1$, it follows that there exists a section in~$\Ind\cC$ to the projection $B\tto B/B^+\simeq \unit$. Inducing via $A\otimes_B-$ thus produces a section $A/AB^+\to A$ to $A\tto A/AB^+$. From this section, we obtain the (left) $B$-module morphism
$$B\otimes A/AB^+\;\to\; A,$$
which is clearly an epimorphism modulo $B^+$. By applying the property in \ref{ThmNak}(1)(b), which is satisfied by Corollary~\ref{NPNak}, we thus obtain an epimorphism of $B_{B^+}$-modules from $B_{B^+}\otimes A/AB^+$ to $A_{B^+}$ which fits into a commutative square
$$\xymatrix{
B_{B^+}\otimes A/AB^+\ar@{->>}[rr]\ar@{^{(}->}[d]&& A_{B^+}\ar[d]\\
A_{B^+}\otimes A/AB^+\ar@{->>}[rr]&& A_{B^+}\otimes_{B_{B^+}} A_{B^+}\simeq A_{B^+}\otimes_B A,
}$$
where the second row is obtained from the first by inducing from $B_{B^+}$ to $A_{B^+}$. Since~\eqref{EqAB} is an isomorphism, the epimorphism on the second row is an epimorphism from a free module to itself. By Theorem~\ref{ThmFree}, the second row is therefore an isomorphism. Since the left arrow is a monomorphism (by exactness of localisation), it follows that the first row is also an isomorphism, proving that $A_{B^+}$ is a free $B_{B^+}$-module and hence flat over $B$.
\end{proof}

\subsection{The Hilbert basis property}
We define the Hilbert basis property following \cite{Ve}.
A module $M$ over an algebra $A\in\Alg\cC$ is {\bf noetherian} if it satisfies the ascending chain condition on submodules or equivalently if every submodule is finitely generated.

\begin{definition}\label{DefNoeth}
An algebra $A\in\Alg\cC$ in a pretannakian category $\cC$ is {\bf noetherian} if one of the following equivalent conditions is satisfied:
\begin{enumerate}
\item Every finitely generated $A$-module is noetherian.
\item Every $A\otimes X$ for $X\in \cC$ is noetherian as an $A$-module.
\item Every $A\otimes L$ for simple $L\in \cC$ is noetherian as an $A$-module.
\end{enumerate}
\end{definition}
Equivalence of these properties follows from the standard observation that an extension of two modules is noetherian if and only if the two modules are both noetherian.

\begin{remark}\label{RemNoe}
\begin{enumerate}
\item The above definition of noetherian algebras is currently not known to be different from the weaker definition which only imposes that the regular module be noetherian.
\item A finitely generated algebra in~$\Alg_k$ remains noetherian as an algebra in~$\Alg\cC$.
\end{enumerate}
\end{remark}

The following potential property of a pretannakian category $\cC$ is clearly hereditary.
\begin{enumerate}
\item[{\bf (HBP)}] Every finitely generated algebra $A\in\Alg\cC$ is noetherian.
\end{enumerate}

\begin{lemma}\label{LemNoe}
The following conditions are equivalent on a pretannakian category $\cC$:
\begin{enumerate}
\item $\cC$ satisfies {\bf(HBP)};
\item For every semisimple object $X\in\cC$, the algebra $\Sym X$ is noetherian.
\end{enumerate}
\end{lemma}
\begin{proof}
Obviously (1) implies (2). Since all quotients of noetherian algebras are again noetherian, it suffices to prove that condition (2) implies that $\Sym Y$ is noetherian for every $Y\in\cC$. We will use characterisation~\ref{DefNoeth}(3).

Take a Loewy filtration of $Y$, which induces a filtration on $\Sym Y$. We denote the associated graded objects by $\gr Y$ and $\gr \Sym Y$. Each chain of submodules in~$(\Sym Y)\otimes L$, for some simple $L\in\cC$, induces one in~$(\gr \Sym Y)\otimes L$ (where one stabilises if and only if the other does). It is therefore sufficient to show that $\gr \Sym Y$ is noetherian.

Now $\gr \Sym Y$ is a quotient of $\Sym(\gr Y)$, where the latter is noetherian by assumption~(2). This concludes the proof.
\end{proof}

The case $\cC=\Ver_p$ of the following lemma was observed in \cite{Ve}.
\begin{lemma}\label{HBT}
If $\cC$ satisfies {\bf (MN1)}, then $\cC$ satisfies {\bf (HBP)}.
\end{lemma}
\begin{proof}
By Lemma~\ref{LemNoe} it suffices to show that for $X$ semisimple, the algebra $\Sym X$ is noetherian. In this case 
$$\Sym X\;\cong\; \Sym(\unit^n\oplus X_1)\;\cong\;k[x_1,\cdots, x_n]\otimes A,$$
where $X_1$ is a direct sum of simples which are not isomorphic to $\unit$, so that $A$ is a finite algebra. This is clearly a noetherian algebra by Remark~\ref{RemNoe}(2).
\end{proof}

\begin{example}
The typical example of a tensor category which does {\em not} satisfy {\bf (HBP)} is $\cC=(\Rep GL)_t$ for $k=\mC$ and $t\not\in\mZ$. Indeed, consider $A\in\Alg\cC$ as in Remark~\ref{Counter}(2). The pentagonal number recurrence relation for the partition function $p(n)$ implies that 
$$p(n)\;>\; p(n-2)+p(n-12)+p(n-15)+p(n-35)+\cdots.$$
Consequently, we can construct an infinite strictly ascending chain of ideals in~$A$ which are obtained iteratively by adding a generator in~$A^0$ in degree $2,12,15$ etc.
\end{example}



\section{Affine group schemes}\label{SecGrp}

\subsection{Groups and subgroups}

\subsubsection{}
An {\bf affine group scheme} in~$\cC$ is a representable functor from $\Alg\cC$ to the category of groups. Equivalently an affine group scheme can be interpreted as a group object in~$(\Alg\cC)^{\op}$. We typically denote an affine group scheme in~$\cC$ by $G$ and its representing commutative Hopf algebra in~$\Ind\cC$ by $\cO(G)$. We use the standard notation $\mu,\eta, \Delta,\varepsilon$ and $S$ for the multiplication, unit, comultiplication, counit and antipode on $\cO(G)$.

\subsubsection{}
A {\bf Subgroup} of $G$ is a subfunctor of
$$G:\,\Alg\cC\;\to\; \Grp$$ which correspond to quotient (Hopf) algebra of $\cO(G).$
A subgroup is {\bf normal} if the corresponding subfunctor yields normal subgroups.

\begin{lemma}\label{LemSub}
Assume that {\bf (Hopf)} holds in~$\cC$. Then every monomorphism $H\to G$ of affine group schemes ({\it i.e.} the kernel is trivial in the terminology of~\ref{DefKer}) corresponds to a subgroup.
\end{lemma}
\begin{proof}
Denote by $I<G$ the affine group scheme represented by the image of $\cO(G)\to \cO(H)$. It suffices prove that $H\to I$ is an isomorphism.

By construction, $H\to I$ is still a monomorphism, in other words $\cO(I)\hookrightarrow \cO(H)$ is an epimorphism in~$\Alg\cC$, and by {\bf (Hopf)} that same morphism is faithfully flat. The conclusion thus follows from Remark~\ref{epifp}.
\end{proof}

\begin{remark}
An alternative formulation of Lemma~\ref{LemSub} is that a morphism $A\to B$ of commutative Hopf algebras is an epimorphism in~$\Alg\cC$ if and only if it is an epimorphism in~$\Ind\cC$.
\end{remark}

\begin{example} \label{ExGLX} Fix $X\in\cC$.
\begin{enumerate}
\item 
The general linear group $GL_X$ is given by the functor
$$A\mapsto GL_X(A):=\Aut_A(A\otimes X).$$
It is represented by a quotient of
$\Sym(X^\vee\otimes X\oplus X\otimes X^\vee).$
\item Denote by $\mA_X$ the group functor
$$A\mapsto \Hom(\unit, A\otimes X),$$
with the additive group structure on the right-hand side. It is represented by $\Sym(X^\vee)$.
\end{enumerate}
\end{example}


\subsection{Quotients and kernels}

\subsubsection{}\label{DefKer} The {\bf kernel} $N$ of a homomorphism $f:G\to H$ is defined as the equaliser in the category of affine group schemes of $f:G\to H$ and $G\to\ast\to H$. That this exists follows from the fact that
the coequaliser of
$$\cO(H)\rightrightarrows \cO(G)$$
can be computed in~$\Alg\cC$ and yields a representing quotient Hopf algebra (it is given by the quotient of $\cO(G)$ by the ideal generated by $\cO(H)^+=\ker\varepsilon$). A consequence of this observation is that
$N(A)$ is the kernel of $G(A)\to H(A)$ for all $A\in\Alg\cC$. In particular the kernel is normal.

A {\bf factor group} of $G$ is a homomorphism $G\to Q$ such that the corresponding algebra morphism $\cO(Q)\to \cO(G)$ is faithfully flat. If {\bf (Hopf)} is satisfied then of course factor groups of $G$ are in bijection with Hopf subalgebras of $\cO(G)$. 

For a factor group $f:G\to Q$, we denote the kernel of $f$ by $N_Q$.

\begin{theorem}\label{ThmQN}
Assume that $\cC$ satisfies {\bf(Hopf)}. Then, for any affine group scheme~$G$ in~$\cC$, the above association yields a bijection between the set of kernels and the set of (isomorphism classes of) factor groups
$$\xymatrix{\{\mbox{factor groups}\}\ar@/^/[rrr]^{(G\to Q)\mapsto N_Q}&&&\{\mbox{kernels}\}\ar@/^/[lll]^{(N\lhd G)\mapsto G/N},
}$$
with inverse defined in \ref{DefQuo} below.
\end{theorem}

\subsubsection{}\label{DefQuo} For a normal subgroup $N$ of an affine group scheme $G$ in~$\cC$, we define the {\bf quotient}~$G/N$ as the coequaliser (in the category of affine group schemes in~$\cC$) of $N\rightrightarrows G$ (considering the inclusion and $N\to\ast\to G$). We can equivalently express this as the coequaliser of $G\times N\rightrightarrows G$. The latter demonstrates existence of the quotient, since the equaliser
$$\cO(G)^N\to\cO(G)\rightrightarrows \cO(G)\otimes \cO(N)$$
in~$\Alg\cC$ (even in~$\Ind\cC$) yields a Hopf subalgebra of $\cO(G)$. 

Under condition {\bf (Hopf)}, quotients are factor groups. Hence the following proposition implies Theorem~\ref{ThmQN}.

\begin{prop}\label{PropQN}
\begin{enumerate}
\item Consider a factor group $f: G\to Q$. Then
$$G/N_Q\;\simeq\;Q,$$ and $Q$ is the sheafification of $A\mapsto G(A)/N_Q(A)$ in the fpqc topology.
\item For any normal subgroup $N\lhd G$, we have 
$$N\;<\; N_{G/N}\;<\; G,$$
where the left inclusion is an equality in case $N$ is a kernel.
\end{enumerate}
\end{prop}
\begin{proof}We prove part (1).
By \ref{DefKer}, the assignment $A\mapsto G(A)/N_Q(A)$ yields a subfunctor of~$Q$. 
 By Lemma~\ref{LemDodu}, it is therefore sufficient to prove that for every algebra $A\in\Alg\cC$ and every element in~$Q(A)$, there exists a faithfully flat $A$-algebra~$B$ with an appropriate element in~$G(B)$.
For an algebra morphism $\cO(Q)\to A$, consider therefore the commutative diagram
\begin{equation}\label{eq724}\xymatrix{
\cO(Q)\ar@{^{(}->}[rr]\ar[d]&&\cO(G)\ar[d]\\
A\ar[rr]&& A\otimes_{\cO(Q)} \cO(G).
}\end{equation}
The upper horizontal arrow is faithfully flat by assumption, so the lower horizontal arrow is also faithfully flat. This proves part (1).

Now we prove part (2). We have a commutative diagram in the category of affine group schemes in~$\cC$:
$$\xymatrix{
N\ar[rr]\ar[d]&& G\ar[d]\\
\ast\ar[rr]&& G/N.
}$$
By definition, this is a pushout diagram. Since $N_{G/N}$ is defined as the fibre product of the two arrows to $G/N$ it follows that $N\to G$ factors through $N_{G/N}\to G$. Also the equality in case $N$ is a kernel follows from this diagram.
\end{proof}

\subsubsection{} We define a {\bf short exact sequence} of affine group schemes to be a sequence
$$1\to N\to G\to Q\to 1$$
where $G\to Q$ is a factor group and $N$ its kernel. Under condition {\bf (Hopf)} this is equivalent, by Theorem~\ref{ThmQN}, to demanding that $N$ is a kernel and $Q$ its quotient.

\begin{porism}\label{CorGN}
Assume that $\cC$ is {\bf MN} and {\bf GR} and that $k$ is algebraically closed. Consider a
short exact sequence $N\to G\to Q$ where $\cO(G)$ is finitely generated, then 
$$1\to N(\unit)\to G(\unit)\to Q(\unit)\to 1 $$
is exact.
\end{porism}
\begin{proof}
We consider the commutative diagram \eqref{eq724}, for $A=\unit$ so that $\cO(Q)\to\unit$ represents an element of $Q(\unit)$. The algebra in the lower right corner $\unit\otimes_{\cO(Q)}\cO(G)$ is finitely generated, so any quotient by a maximal ideal must be isomorphic to $\unit$ by Proposition~\ref{Prop4Bas}. The corresponding composite $\cO(G)\to \unit$ demonstrates that $G(\unit)\to Q(\unit)$ is indeed surjective.
\end{proof}

\begin{corollary}
Assume that $\cC$ satisfies {\bf (Hopf)} and consider an affine group scheme $G$ in~$\cC$ with kernels $N_1\lhd G \rhd N_2$ with $N_1\cap N_2=1$.
 We have an isomorphism 
$$G\;\xrightarrow{\sim}\; (G/N_1)\times_{G/N_1N_2}(G/N_2)$$

\end{corollary}
\begin{proof}
One can copy the proof of \cite[Proposition~6.3.4]{Tann} by Proposition~\ref{PropQN}(1). 
\end{proof}

Note that we did not address the following natural question:
\begin{question}\label{RemChev}
Is every normal subgroup of an affine group scheme in an {\bf MN} and {\bf GR} pretannakian category a kernel?
\end{question}

The main difficulty in proving that all normal subgroups are kernels is the absence of Chevalley's Theorem, see \cite[\S 4.h]{Milne}. This is already problematic in~$\sVec$, as one can not use exterior powers to `create invertible objects'. That all normal subgroups are kernels for $\cC=\sVec$ was proved in \cite{Ma, Zu}.

\subsection{Representations and algebraic groups}

\subsubsection{}For an affine group scheme $G$ in  $\cC$, we denote by $\Rep_{\cC}G$ the category of $\cO(G)$-comodules in~$\cC$. We can view the objects equivalently as pairs of $(X,\rho)$, with $X\in \cC$ and $\rho:G\to GL_X$ a morphism of affine group schemes in~$\cC$. By abuse of notation we will also use~$\rho$ for the associated coaction $X\to\cO(G)\otimes X$. Then $\Rep_{\cC}G$ is clearly again a tensor category, equipped with a (forgetful) tensor functor to $\cC$.


\begin{lemma}\label{LemIndRep}
\begin{enumerate}
\item The category of $\cO(G)$-comodules in~$\Ind\cC$ is equivalent to $\Ind\Rep_{\cC}G$.
\item For an $\cO(G)$-comodule $X$ in~$\Ind\cC$ and a subobject $\cC\ni X_0\subset X$, there exists a subcomodule $X_1\subset X$ with $X_0\subset X_1\in\cC$.
\end{enumerate}
\end{lemma}
\begin{proof}
Part (1) is an immediate consequence of part (2). For part (2) we start from the composition of the inclusion and coaction
$$X_0\hookrightarrow X\to \cO(G)\otimes X.$$
We write $\cO(G)$ as $\varinjlim H_\alpha$ for $\cC\ni H_\alpha\subset \cO(G)$. The above composite factors through $H_\alpha\otimes X\to\cO(G)\otimes X$ for some $\alpha$. We define $X_1$ as the image of the corresponding
$$H_\alpha^\vee\otimes X_0\to X.$$
It follows easily that $X_0\subset X_1$ and that $X_1$ does not depend on the choice of $\alpha$.

Using coassociativity we can then observe that the composite
$$H_\alpha^\vee\otimes X_0\tto X_1\to \cO(G)\otimes X$$
of the defining epimorphism and the coaction, lands in~$\cO(G)\otimes X_1$. Consequently, $X_1\subset X$ is a subrepresentation.
\end{proof}

\begin{lemma}\label{LemAlgG}
For an affine group scheme $G$ in~$\cC$ the following conditions are equivalent:
\begin{enumerate}
\item The algebra $\cO(G)$ is finitely generated;
\item $G$ is a subgroup of $GL_X$ for some $X\in\cC$.
\end{enumerate}
We call such affine group schemes {\bf algebraic groups}.  If $\cC$ satisfies {\bf (Hopf)}, the conditions are furthermore equivalent to:
\begin{enumerate}
\item[(3)] The tensor category $\Rep G$ is finitely generated over $\cC\subset \Rep G$.
\end{enumerate}

\end{lemma}
\begin{proof}
Condition (2) implies that $\cO(G)$ is a quotient algebra of $\cO(GL_X)$, where the latter is finitely generated.

If condition (1) is satisfied, we can take $\cC\ni Y\subset \cO(G)$ which generates the algebra $\cO(G)$. By Lemma~\ref{LemIndRep}, we can replace $Y$ by a subrepresentation $\cC\ni X\subset \cO(G)$ in which it is contained. By adjunction applied to the co-action, we find a morphism
$$X\otimes X^\vee\to \cO(G).$$
Denote by $\beta:\unit\to X^\vee$, the adjoint of the composite $X\hookrightarrow \cO(G)\xrightarrow{\epsilon}\unit$. It then follows that the precomposition with $X\otimes \beta$ of the displayed morphisms is just the inclusion of $X$ into $\cO(G)$. In particular, the image of the displayed morphism contains $X$. It follows that $\cO(GL_X)\to \cO(G)$ is an epimorphism (in~$\Ind\cC$), which implies that (2) is satisfied.

Now we prove that (2) implies (3). Since every representation is a subrepresentation of $\cO(G)\otimes V$ for some $V\in\cC$, it follows easily that $\Rep\cC$ is generated over $\cC$ by an $X$ as in (2).

Now assume that {\bf (Hopf)} holds. Under condition (3) there exists a $G$-representation $X$ which generates $\Rep G$ together with $\cC$. Let $N$ be the kernel of $G\to GL_X$. Then $N$ acts trivially on any $G$-representation (including $\cO(G)$). It follows that $N$ is trivial, so that $G\to GL_X$ is a monomorphism. By Lemma~\ref{LemSub}, condition (2) follows.
\end{proof}

\begin{corollary}\label{CorProAlg}
Every affine group scheme $G$ is isomorphic to $\varprojlim G_\alpha $ for an inverse system of algebraic groups $G_\alpha$. 
\end{corollary}
\begin{proof}
Consider all representations $G\to GL_X$, $X\in\cC$. Taking the image of $\cO(GL_X)\to \cO(G)$ yields a subgroup $G'<GL_X$ through which the morphism factors. Moreover, by the proof of Lemma~\ref{LemAlgG}, the subalgebra $\varinjlim\cO(G')$ of $\cO(G)$ will be the entire algebra.
\end{proof}

\subsection{Some constructions}
We focus on some constructions of affine group schemes associated to an affine group scheme $G$ in~$\cC$ and a representation $(X,\rho)\in\Rep_{\cC}G$.

\subsubsection{} 
We define the semidirect product $G\ltimes \mA_X$ (represented by $\cO(G)\otimes \Sym X^\vee$),
with obvious group law using the defining $\rho(A):G(A)\to \Aut_A(A\otimes X)$.

\subsubsection{Stabiliser}\label{DefStab} Consider a monomorphism (in~$\cC$)
$$\nu:\unit\hookrightarrow X.$$
We have the corresponding morphisms $ \eta\otimes \nu$ and $\rho\circ \nu$
$$\unit\rightrightarrows \cO(G)\otimes X.$$
Denote by $\cO(G_\nu)$ the quotient algebra of $\cO(G)$ with respect to the ideal corresponding to the relation $X^\vee\rightrightarrows \cO(G)$. This corresponds to a subgroup $G_\nu<G$ given by
$$G_\nu(A)\;=\; \{g\in G(A)\,|\,\rho(g)\circ(A\otimes\nu) = A\otimes\nu\}.$$
In particular, $G_\nu$ is the maximal subgroup $H<G$ for which $\nu$ is a morphism in~$\Rep_{\cC}H$.

We have the following standard alternative realisation of stabiliser subgroups.
\begin{lemma}\label{LemStabEq}
We have an equaliser diagram
$$G_\nu\;\to\;G\;\rightrightarrows \; G \ltimes \mA_X$$
in the category of affine group schemes in~$\cC$.
\end{lemma}

\subsubsection{Transporter} For a subobject $U\subset X$ in~$\cC$, we have the corresponding composite morphism
$$U\hookrightarrow X\xrightarrow{\rho}\cO(G)\otimes X\tto \cO(G)\otimes X/U.$$
Denote by $\cO(G_U)$ the quotient of $\cO(G)$ with respect to the ideal generated by the image of the corresponding morphism
$$U\otimes (X/U)^\vee\;\to\; \cO(G).$$
It follows that $G_U$ is the maximal subgroup $H<G$ for which $U$ is a subobject of $X$ in~$\Rep_{\cC} H$.

The most natural case of the following lemma is $H=G_U$, implying $G_U=G_\nu$, but we will need it in the following generality.

 \begin{lemma}\label{LemTransStab}
Consider $U\subset X$ in~$\cC$ and a subgroup $H<G_U$. If $\Res^{G_U}_HU^\vee\subset \Res^{G}_{H}Y$ (as $H$-representations) for some $Y\in \Rep G$, then $H<G_\mu<G_U$ for the composite in $\cC$
 $$\mu: \unit\xrightarrow{\co_U} U\otimes U^\vee\hookrightarrow X\otimes Y.$$

 \end{lemma}
\begin{proof}
Since $\co_U$, $U\hookrightarrow X$ and $U^\vee\hookrightarrow Y$ are morphisms in~$\Rep H$, it follows by definition of~$G_\mu$ that $H<G_\mu$.

To show that $G_\mu<G_U$ it suffices to show that
$$U\hookrightarrow X\xrightarrow{c}\cO(G_\mu)\otimes X\stackrel{\cO(G_\mu)\otimes\pi}{\tto} \cO(G_\mu)\otimes (X/U)$$
composes to $0$. Here and below, we denote any co-action by $c$ write $\pi:X\tto X/U$. From the definition of $\cO(G_\mu)$, it follows in particular that 
$$\unit \xrightarrow{\co_U} U\otimes U^\vee\hookrightarrow X\otimes Y\xrightarrow{c} \cO(G_\mu)\otimes X\otimes Y\stackrel{\cO(G_\mu)\otimes \pi\otimes Y}{\tto} \cO(G_\mu)\otimes (X/U)\otimes Y$$
composes to zero. We can reorder this composition as the composition of 
$$\unit \xrightarrow{\co_U} U\otimes U^\vee\hookrightarrow X\otimes U^\vee\xrightarrow{c \otimes U^\vee}\cO(G_\mu)\otimes X\otimes U^\vee\stackrel{\cO(G_\mu)\otimes\pi\otimes U^\vee}{\tto}\cO(G_\mu)\otimes (X/U)\otimes U^\vee$$
followed by a monomorphism coming from $U^\vee\hookrightarrow Y$ and an isomorphism coming from the coaction isomorphism $\cO(G_\mu)\otimes Y\xrightarrow{\sim}\cO(G_\mu) \otimes Y$. Hence the last displayed composite must be zero, which is equivalent to the desired composition being zero.
\end{proof}

We conclude this section by observing that under some natural assumptions all subgroups are transporters.
\begin{lemma}\label{LemTransp}
Assume that $G$ is an algebraic group in~$\cC$ and $H<G$ a subgroup corresponding to a finitely generated ideal $J$ in~$\cO(G)$. Then $H=G_U$ for some $(Z,\rho)\in \Rep_{\cC}G$ and $\cC\ni U\subset Z$.
\end{lemma}
\begin{proof}
By assumption, there is a generator $\cC\ni Z_0\subset \cO(G)$ so that $J\cap Z_0$ generates $J$. We choose $Z=Z_1\subset\cO(G)$ a subrepresentation containing $Z_0$, guaranteed to exist by Lemma~\ref{LemIndRep}, and $U=Z\cap J$. That $H=G_U$ follows as in the classical case, see \cite[II.\S2, no 3]{DG}.
\end{proof}

\subsection{Observable and epimorphic subgroups}\label{SecObs}
Consider an affine group scheme $G$ in~$\cC$ with subgroup $H$. We denote the restriction functor by
$$\Phi=\Res^G_H:\Rep_{\cC}G\to \Rep_{\cC}H.$$
We have
$$\Phi_\ast=\Ind^G_H:\Rep^\infty_{\cC} H\to\Rep^\infty_{\cC}G$$
with
$$\Ind^G_HX:=\mathrm{Eq}(\cO(G)\otimes X\rightrightarrows\cO(G)\otimes \cO(H)\otimes X).$$
Here, one arrow is given by $\cO(G)\otimes \rho$ for the coaction $\rho$ and the other by $\Delta\otimes X$ followed by projection onto $\cO(H).$ In particular, we find that $\Ind^G_H\unit$ is given by the right $H$-invariants $\cO(G)^H$ in~$\cO(G)$.


\begin{prop}\label{PropEpi}
The following conditions on $H<G$ are equivalent:
\begin{enumerate}
\item $H\to G$ is an epimorphism in the category of affine group schemes in~$\cC$.
\item $\Phi$ is full on isomorphisms.
\item $\Phi$ is full.
\item $\cO(G)^H=\unit$.
\item If $H$ is contained in a stabliser subgroup $G_\nu<G$, then $G_\nu=G$.
\end{enumerate}
Such subgroups are called {\bf epimorphic}.
\end{prop}
\begin{proof}
First we show that (3) and (4) are equivalent. By duality and adjunction, property~(3) is equivalent to
$$\Hom_G(X,\unit)\to \Hom_G(X,\Ind^G_H\unit)$$
being an isomorphism for all $X\in \Rep_{\cC}G$. This is equivalent to condition in (4).

Obviously (3) implies (2).

Now we prove that (2) implies (1). Under condition (2) it follows easily that for a diagram $H\to G\rightrightarrows GL_X$ (with $X\in\cC$) with the two composites equal, the two morphisms $G\rightrightarrows GL_X$ have to be equal. It follows by Lemma~\ref{LemAlgG} that we can replace $GL_X$ by an arbitrary algebraic group and keep the same conclusion. We can then extend to arbitrary affine group schemes by Corollary~\ref{CorProAlg}.


Finally, that (1)$\Rightarrow$(5)$\Rightarrow$(3) can be proved as in \cite[Section~2]{Br}. Indeed, assume that $\Phi$ is not full. This means there is a $G$-representation $(X,\rho)$ with some morphism $\nu:\unit\hookrightarrow X$ in~$\cC$ which is $H$-linear but not $G$-linear. Consequently, $H$ is a subgroup of $G_\nu<G$. Thus (5) implies (3). Since stabilisers are equalisers, see Lemma~\ref{LemStabEq}, subgroups of them cannot yield epimorphisms to $G$, showing that (1) implies (5).
\end{proof}

%

\begin{definition}\label{DefObs}
The subgroup $H<G$ is called {\bf observable} if $\Phi$ satisfies the equivalent conditions in  Proposition~\ref{Propadj0}, {\it i.e.} if either
\begin{enumerate}
\item $\Ind^G_H$ is faithful; or
\item Every $X\in \Rep_{\cC}H$ is a subobject of $\Res^G_HY$ for some $Y\in\Rep_{\cC}G$.
\end{enumerate}
\end{definition}

\begin{prop}\label{PropObsMono}
For a subgroup $H<G$, we have the following implications between the properties below:
$$(1)\;\Rightarrow\; (2)\;\Leftrightarrow \; (3)\;\Rightarrow (4).$$
\begin{enumerate}
\item $H<G$ is observable;
\item $H$ is the intersection of a family of stabiliser subgroups $G_\nu<G$;
\item $H<G$ is the maximal subgroup which acts trivially on the subobject ${}^H\cO(G)$ of the left regular representation $\cO(G)$.
\item $H\to G$ is a regular monomorphism in the category of affine group schemes in~$\cC$.
\end{enumerate}
\end{prop}
\begin{proof}
To prove that (1) implies (2) consider the ideal $J<\cO(G)$ which defines $H$. By Lemma~\ref{LemIndRep}, we can take a family of compact subrepresentations $X_i\subset \cO(G)$ such that every $\cC\ni X\subset \cO(G)$ is contained in some $X_i$. We also set $U_i=X_i\cap J$. We claim that $H=\cap_i G_{U_i}$. Indeed, $H<G_{U_i}$ follows by definition. On the other hand, for $A\in\Alg\cC$ consider an element $a\in\cap_i G_{U_i}(A)$. For the corresponding homomorphism $a:\cO(G)\to A$, the composition
$$U_i\xrightarrow{\Delta} \cO(G)\otimes \cO(G)\xrightarrow{a\otimes p} A\otimes \cO(H)$$
is zero, for $p$ the projection onto $\cO(G)/J=\cO(H)$. Composing with $A\otimes\varepsilon$ thus shows that~$a$ restricts to zero on every $U_i$ and hence on $J$. Therefore $a\in H(A)$, proving the claim.
Now we can apply Lemma~\ref{LemTransStab} to replace the transporter subgroups $G_{U_i}$ by stabiliser subgroups $H<G_{\nu_i}<G_{U_i}$, so we conclude indeed $H=\cap_i G_{\nu_i}$.

Now we prove that (2) implies (4). By Lemma~\ref{LemStabEq}, every $G_\nu\to G$ is regular. For a collection of equalisers $G_i\to G\rightrightarrows G_i'$, we have an equaliser
$$\cap_i G_i\;\to\; G\;\rightrightarrows\; \prod_i G_i'$$
which concludes the proof.

Finally we prove equivalence between (2) and (3). Every $X\in \Rep_{\cC}G$ can be embedded in a representation of the form $\cO(G)\otimes Y$, for some $Y\in\cC$, with action derived from the left regular action. By using adjunction, we can then see that (2) is equivalent to the condition that $H$ is the maximal subgroup which acts trivially on a certain collection of subobjects $\cC\ni Z\subset \cO(G)$ (automatically in~${}^H\cO(G)$). This is in turn equivalent to (3), by Lemma~\ref{LemIndRep}.
\end{proof}


Consider the following potential property on $\cC$:
\begin{enumerate}
\item[ {\bf (Obs)}] The conditions (1)-(4) in Proposition~\ref{PropObsMono} are equivalent for every subgroup of every affine group scheme in~$\cC$.
\end{enumerate}
We will prove after some preparation that this property is hereditary.

\begin{lemma}\label{LemInvSys}
Consider an inverse system $\{H_i<G_i\,|\, i\in I\}$ of subgroups (for $j>i$, the morphism $H_j<G_j\to G_i$ lands in~$H_i$). If each $H_i<G_i$ is epimorphic (resp. observable), then $\varprojlim H_i<\varprojlim G_i$ is epimorphic (resp. observable).
\end{lemma}
\begin{proof}
This follows easily from the identification $\Rep_{\cC}G\simeq\varinjlim \Rep_{\cC}G_i$ and the formulation of the conditions in \ref{PropEpi}(3) and \ref{DefObs}(2).
\end{proof}

\begin{theorem}\label{ThmObs0}
The following three conditions on $\cC$ are equivalent:
\begin{enumerate}
\item For every subgroup $H<G$ of every affine group scheme $G$ in~$\cC$, there exists an intermediary group $H<H'<G$ such that $H<H'$ is epimorphic and $H'<G$ observable;
\item The conditions \ref{PropObsMono}(1) and \ref{PropObsMono}(2) are equivalent whenever $G$ is algebraic.
\item $\cC$ satisfies {\bf (Obs)}.
\end{enumerate}
If these properties are satisfied, then the intermediary group in (1) is unique and is given by the intersection of all stabiliser subgroups of $G$ that contain~$H$, or alternatively as the maximal subgroup of $G$ which acts trivially on ${}^H\cO(G)\subset \cO(G)$.
We refer to it as the {\bf observable hull} of $H$ in~$G$.
\end{theorem}
\begin{proof}
If for a chain of subgroups $H<H'<G$, we have that $H\to H'$ is an epimorphism while $H'\to G$ is a regular monomorphism, then $H'$ is unique with these properties.

To show that (1) implies (3), we only need to show that \ref{PropObsMono}(4) implies \ref{PropObsMono}(1) under our assumption (1). Consider thus a subgroup $H<G$ for which $H\to G$ is a regular monomorphism.  Since $H'\to G$ is also a regular monomorphism, by \autoref{PropObsMono}, the uniqueness in the first pargraph implies $H=H'$, so $H<G$ is observable.

Obviously (3) implies (2).

Now we prove that (2) implies (1). We only need to show that (1) is satisfied for algebraic groups $G$, since the general case then follows quickly from Corollary~\ref{CorProAlg} and Lemma~\ref{LemInvSys}. Consider $H<H'<G$ with $G$ algebraic and with $H'$ defined as the intersection of all stabiliser subgroups of $G$ that contain~$H$. By assumption (that \ref{PropObsMono}(2) implies \ref{PropObsMono}(1)) $H'<G$ is observable. It thus suffices to show that $H<H'$ is epimorphic. We do this (using Proposition~\ref{PropEpi}) by showing that $\Rep H'\to \Rep H$ is full. Consider therefore $X\in \Rep H'$ and $\nu:\unit\hookrightarrow X$ which is $H$-linear. We must show it is in fact $H'$-linear. Since $H'<G$ is observable, we can in fact assume without loss of generality that $X$ is a $G$-representation. But then by construction $H<G_\nu $, which implies by definition of $H'$ that $H'<G_\nu$ and $\nu$ is indeed $H'$-linear.

We prove the final statements in the theorem. Uniqueness follows from the first paragraph of the proof. The first explicit description of $H'$ follows from the previous paragraph. The second description then follows from the argument in the last paragraph in the proof of Proposition~\ref{PropObsMono}.
\end{proof}

\begin{corollary}
Property {\bf (Obs)} is hereditary.
\end{corollary}
\begin{proof}
We can instead prove that the condition in Theorem~\ref{ThmObs0}(2) is hereditary. So concretely, consider a tensor functor $\cC\to \cC_1$ to a pretannakian category $\cC_1$ for which property \ref{PropObsMono}(2) always implies \ref{PropObsMono}(1). 

We consider a pair $H<G$ of affine group schemes in~$\cC$ and assume that it satisfies \ref{PropObsMono}(2). We denote by $H_1<G_1$ the same groups considered in~$\cC_1$. Clearly \ref{PropObsMono}(2) is still satisfied for this pair, and hence $H_1<G_1$ is observable, or in other words $\Ind^{G_1}_{H_1}$ is faithful. It follows immediately that also $\Ind^{G}_{H}$ is also faithful.
\end{proof}

Finally, we point out a connection between observability and the question of whether every normal subgroup is a kernel.

\begin{prop}\label{PropNObs}
Assume that $\cC$ satisfies {\bf (Hopf)}, then the following properties are equivalent on a normal subgroup $N\lhd G$:
\begin{enumerate}
\item $N\lhd G $ is observable;
\item $N\lhd G$ is a kernel;
\end{enumerate}
\end{prop}
\begin{proof}
First we prove that (1) implies (2).
 Proposition~\ref{PropObsMono} implies that there exists a collection of representations $X_i$ of $G$ with $\{\nu_i:\unit\hookrightarrow X_i\}$ so that $N=\cap_i G_{\nu_i}$. For each $i$ let $Y_i=X_i^N\subset X_i$ be the maximal $G$-subrepresentation on which $N$ acts trivially. In particular, the image of $\nu_i$ is contained in~$Y_i$. It then follows that $N$ is the kernel of
$$G\;\to\;\prod_i GL(Y_i).$$
Indeed, $N$ is included in the kernel by construction of $Y_i\subset X_i$. Moreover, the kernel is clearly contained in~$\cap_i G_{\nu_i}$. This shows that (2) is satisfied.

Now we show that (2) implies (1). We set $Q=G/N$ and will prove that $\Ind^G_N$ is faithful (and exact). As the argument is very classical we only sketch it. It follows easily that $\Ind^G_NX$ inherits an $\cO(Q)$-module structure, from the $\cO(G)$-module structure on $\cO(G)\otimes X$. By {\bf (Hopf)}, it is sufficient to prove that $\cO(G)\otimes_{\cO(Q)}\Ind^G_N-$ is faithful. From the isomorphism
$$G\times_QG\;\simeq\; G\times N$$
one derives that
$\cO(G)\otimes_{\cO(Q)}\Ind^G_N-\;\simeq\; \cO(G)\otimes -,$
from which the conclusion follows.
\end{proof}

\appendix

\section{Classical (strong) observability}\label{App}

\subsection{Observability}
In this section we demonstrate how our treatment of observability in Section~\ref{SecObs} already leads to slightly stronger results for ordinary affine group schemes over fields than seem available in the literature of \cite{BBHM, Gr}.

\begin{theorem}\label{ClassThm}
Let $k$ be an arbitrary field, and consider an affine group scheme $G$ over~$k$ with subgroup $H<G$. 
\begin{enumerate}
\item The following conditions are equivalent on $H<G$.
\begin{enumerate}
\item Every finite dimensional representation of $H$ extends to one of $G$;
\item The geometric induction functor $\Ind^G_H$ is faithful;
\item The subgroup $H$ is the intersection of a family of stabiliser subgroups of vectors in~$G$-representations;
\item The inclusion $H\to G$ is a regular monomorphism in the category of affine group schemes over $k$.
\end{enumerate}
Furthermore, if $G$ is of finite type, then these properties are equivalent to
\begin{enumerate}
\item[(e)] The quotient scheme $G/H$ is quasi-affine.
\end{enumerate}
\item There exists a unique subgroup $H<H'<G$ for which $H\to H'$ is an epimorphism and $H'< G$ satisfies conditions (a)-(d) in part (1).
\end{enumerate}
\end{theorem}
\begin{proof}
Excluding (1)(e), it follows from Theorem~\ref{ThmObs0} (and Proposition~\ref{PropObsMono}) that it suffices to prove that, for $G$ of finite type, property (1)(c) implies (1)(b).

We will therefore prove that for $G$ of finite type, we have the implications (c)$\Rightarrow$(e)$\Rightarrow$(b) in (1), from which the entire theorem follows.

We prove that (c) implies (e). As $G$ is of finite type, we can assume that $H$ is simply the stabiliser of a vector $v$ in a finite dimensional $G$-representation $V$. It follows from \cite[Proposition~III.3.5.2]{DG} that the morphism of schemes $G/H\to\mA V$ is an immersion, in particular $G/H$ is quasi-affine.

Now we prove that (e) implies (b). For this it suffices to observe (see \cite[Chapter~5]{Jantzen}) that $\Ind^G_H$ (say post-composed with the forgetful functor) can be realised as a composition
$$\Rep_k H\;\xrightarrow{\cL}\; \QCoh(G/H)\;\xrightarrow{\Gamma}\;\Vecc^\infty_k, $$
where $\cL$ is faithful (and exact) and $\Gamma$ is the functor taking global sections. Indeed, the functor $\Gamma$ being faithful is a (defining) property of quasi-affine schemes.
\end{proof}

\begin{remark}
\begin{enumerate}
\item Equivalence of (1)(a), (c) and (e) was proved in \cite{BBHM} for {\em smooth} affine group schemes of {\em finite type} over {\em algebraically closed} fields. The interpretation in (1)(d) seems to be new even in finite type.
\item The subgroup $H'$ from (2) appears in \cite[Chapter~I]{Gr}, again for {smooth} affine group schemes of {finite type} over {algebraically closed} fields.
\item Theorem~\ref{ClassThm}(1) was applied in \cite{CEOP} in the generality provided here (that is, for non-smooth groups).
\item The proof of Theorem~\ref{ClassThm} avoids using the approach of \cite[Theorem~1]{BBHM} focused on 1-dimensional $H$-representations. This makes it more amenable to generalisation to other tensor categories, see also the discussion following Question~\ref{RemChev}.
\item Another interpretation of the observable hull is as follows. In the category of affine group schemes over a field, monomorphisms correspond to subgroups, and one can show that quotient groups correspond to regular epimorphisms. Theorem~\ref{ClassThm} shows that for any group homomorphism $G_1\to G_2$ we have, besides the usual decomposition $G_1\to H\to G_2$ into a regular epimorphism and a monomorphism ($H$ is the image), we also have a (unique) decomposition $G_1\to H'\to G_2$ into an epimorphism and a regular monomorphism.
\end{enumerate}
\end{remark}

\subsection{Strong observability}

For the sake of completeness we also extend the result from \cite{CPS} from smooth groups. Contrary to our treatment of observability, which gives a `new' proof, the result in this section is a simple direct reduction to the result in \cite{CPS}.

We call a subgroup $H<G$ `exact' if the geometric induction functor $\Ind^G_H$ is exact. Due to the results in \cite{CPS} such subgroups are called `strongly observable' in finite type.

\begin{theorem}\label{ThmStrong}
Let $k$ be an algebraically closed field.
Let $G$ be an affine group scheme of finite type over $k$ with a subgroup $H<G$. The following properties are equivalent:
\begin{enumerate}
\item The subgroup $H<G$ is exact;
\item The subgroup $H<G$ is exact and observable;
\item The subgroup $H_{\red}<G_{\red}$ is exact;
\item The quotient scheme $G/H$ is affine.
\end{enumerate}
\end{theorem}
\begin{proof}
The equivalence of these properties in case $G=G_{\red}$ is the main result of \cite{CPS}.

In general, (4) implies (2), due to the last paragraph of the proof of Theorem~\ref{ClassThm}. Obviously (2) implies (1). 

Now we prove that (1) implies (3). Consider the commutative square of induction functors
$$
\xymatrix{
\Rep^{\infty}H_{\red}\ar[rr]\ar[d]&&\Rep^\infty G_{\red}\ar[d]\\
\Rep^{\infty}H\ar[rr]&&\Rep^\infty G.\\
}
$$
By assumption, the lower horizontal arrow is exact. Note that $H/H_{\red}$ and $G/G_{\red}$ are affine (this is well-known but also follows from the same considerations as in the next paragraph). Since (4) implies (2), we thus find that the two vertical arrows are exact and faithful. It then follows that the upper horizontal arrow is also exact indeed.

Finally, we prove that (3) implies (4). By \cite{CPS}, (3) implies that $G_{\red}/H_{\red}$ is affine.
Now we can observe that we have an isomorphism \begin{equation}\label{redquoeq}G_{\red}/H_{\red}\;\simeq\;(G/H)_{\red},\end{equation} see \ref{redquo} below, and $(G/H)_{\red}$ is affine if and only if $G/H$ is affine, see \cite[I.2.6.1]{DG}.
\end{proof}

\subsubsection{}\label{redquo} We prove the isomorphism \eqref{redquoeq}. Restrict the action of $G$ on $G/H$ to $G_{\red}<G$ and observe that this restricts to an action on $(G/H)_{\red}$. Clearly the stabiliser of the origin is 
$$H_{\red}\;=\; H\cap G_{\red}.$$
We thus find an immersion $G_{\red}/H_{\red}\to (G/H)_{\red}$, see \cite[Proposition~III.3.5.2]{DG}. However, this immersion is clearly surjective and hence a closed immersion corresponding to a nil ideal. Since the target is already reduced, this must be an isomorphism.

\subsection*{Acknowledgement}
The author thanks Pavel Etingof, Walter Ferrer Santos and Victor Ostrik for interesting discussions.
The research was partially supported by ARC grant DP210100251.

\subsection*{Conflict of interest}
The author has no conflict of interest to declare that are relevant to this article.

\end{document}